\documentclass[12pt,reqno]{amsart}

\usepackage{tikz}
\usetikzlibrary{positioning, arrows.meta}
\usepackage{epsfig}
\usepackage{amscd}
\usepackage[mathscr]{eucal}
\usepackage{amssymb}
\usepackage{amsxtra}
\usepackage{amsmath}
\usepackage[all]{xy}
\usepackage{mathtools}
\usepackage[pdfencoding=auto]{hyperref}
\usepackage{bookmark}
\usepackage{hyperref}
\usepackage{caption}
\usepackage{colortbl}
\usepackage{arydshln}
\DeclarePairedDelimiter\abs{\lvert}{\rvert}

\theoremstyle{plain}
\newtheorem{mainthm}{Theorem}
\newtheorem{thm}{Theorem}[subsection]
\newtheorem{lem}[thm]{Lemma}
\newtheorem{prop}{Proposition}[section]
\newtheorem{cor}[prop]{Corollary}

\usepackage{amsthm}
\newtheorem{theorem}{Theorem}[section]
\newtheorem{lemma}[theorem]{Lemma}
\newtheorem{conj}[theorem]{Conjecture}
\newtheorem*{myquestion}{Question}

\theoremstyle{definition}
\newtheorem{dfn}[subsection]{Definition}
\newtheorem*{dfnnonum}{Definition}

\theoremstyle{remark}
\newtheorem{rem}[subsection]{Remark}

\newtheorem{ex}[subsection]{Example}

\begin{document}

\bigskip

\title[New disc. for $Z$-function and corrected Gram's law]{A New Discriminant for the Hardy Z-Function and the Corrected Gram's law} 
\date{\today}

\author{Yochay Jerby}

\address{Yochay Jerby, Faculty of Sciences, Holon Institute of Technology, Holon, 5810201, Israel}
\email{yochayj@hit.ac.il}


%
%
% ...
\begin{abstract}
In this paper, we introduce a novel variational framework rooted in algebraic geometry for the analysis of the Hardy $Z$-function. Our primary contribution lies in the definition and exploration of $\Delta_n(\overline{a})$, a newly devised discriminant that measures the realness of consecutive zeros of $Z(t)$. Our investigation into $\Delta_n(\overline{a})$ and its properties yields a wealth of compelling insights into the zeros of $Z(t)$, including the corrected Gram's law, the second-order approximation of $\Delta_n(\overline{a})$, and the discovery of the G-B-G repulsion relation. Collectively, these results provide compelling evidence supporting a new plausibility argument for the Riemann hypothesis.
\end{abstract}

\maketitle
%
%
\section{Introducing the Summary of the Main Results}
\label{s:1}

\subsection{Lack of Plausibility Argument for RH} The Hardy Z-function, denoted as \(Z(t)\), is defined by:
\[
Z(t) = e^{i \theta(t)} \zeta \left ( \frac{1}{2} +it \right )
\]
where \(\theta(t)\) is the Riemann-Siegel \(\theta\)-function, given by the equation:
\[
\theta(t) = \text{arg} \left ( \Gamma \left ( \frac{1}{4} + \frac{i t}{2} \right ) \right ) -\frac{t}{2} \log(t),
\]
see \cite{I}. The Riemann Hypothesis (RH), which conjectures that all non-trivial zeros of the Riemann zeta function have real part \( \sigma = \frac{1}{2} \),  can equivalently be stated that all the zeros of the Hardy Z-function are real. The conjecture stands as one of the most central and enduring challenges in analytic number theory and its resolution holds profound implications for the distribution of prime numbers and many other areas of mathematics.

Despite the Riemann Hypothesis's (RH) deceptively straightforward assertion and its empirical validity, as observed through extensive numerical computations, the foundation behind its truth remains elusive. The depth of its ties to myriad pivotal areas of mathematics, ranging from the distribution of prime numbers to quantum chaos and random matrix theory, further accentuates its enigma. Over the years, many have remarked on the absence of any solid heuristic or plausibility argument underpinning the hypothesis. For instance, Edwards, in his seminal work \cite{E}, emphasizes this gap by noting the absence of any tangible reason to deem the RH as "probable", even though its validity for an extensive range of roots above the real axis seems to hint otherwise. It is worth noting that the latest numerical verifications, conducted by Platt and Trudgian, confirm the RH up to heights of $3\cdot 10^{12}$ above the real axis \cite{PT}. This stark lack of a compelling plausibility argument contributes to the Riemann Hypothesis remaining one of the most tantalizing and resilient unsolved problems in mathematics.

\subsection{Discriminants and Real Zeros} Although the Riemann Hypothesis concerns itself with the zeros of the transcendental function \( Z(t) \) the family of functions for which we can truly offer a full, closed description of their zeros is notably small, encompassing polynomials of degree less than five, certain trigonometric functions, and the like. When considering the Hardy \( Z \)-function and its chaotic nature, any hopes of a 'closed formula' for its zeros become all but impossible.

However, the Riemann Hypothesis doesn't explicitly demand us to \emph{find} the zeros of the \( Z(t) \) function. The actual question is: \emph{are all the zeros of \( Z(t) \) real?}. By drawing an analogy with quadratic functions \( a_2 x^2 + a_1 x + a_0 \), we recognize an important invariant addressing this very question — the discriminant, given by \( \Delta(a_2,a_1,a_0) = a^2_1 - 4a_2 a_0 \). The defining feature of this discriminant is that the zeros of \( a_2 x^2 + a_1 x + a_0 \) are real if and only if \( \Delta(a_2,a_1,a_0) > 0 \). The discriminant captures this essential feature of the zeros, without requiring the direct computation of the zeros. Within the framework of the Riemann Hypothesis, this is \emph{precisely} the kind of perspective we seek. Notably, even though the closed formula for the zeros of the quadratic equation doesn't generalize directly to polynomials of higher degrees, the discriminant's concept can indeed be expanded to polynomials of any degree, as well as systems of polynomial equations and algebraic varieties \cite{GKZ}.

In this work, we introduce an extension of the idea of the discriminant into the transcendental realm of the \( Z(t) \) function. By their very nature, discriminants act as an invariant for a family of functions. Building upon this, for any $N \in \mathbb{N}$ we introduce the novel concept of the \( A \)-parametrized space \( \mathcal{Z}_N \) which is an $N$-dimensional space of variations of $Z(t)$ in the region $2N \leq t \leq 2N+1$. In this study, we introduce, \( \Delta_n(\overline{a}) \),  the local $n$-th discriminant for a pair of consecutive zeros $t_n$ and $t_{n+1}$ within this region. This newly defined discriminant is shown in this work to unveil a wealth of significant new results regarding the zeros of $Z(t)$. In order to describe the construction we need to recall the classical Gram's law, which is central to all that would follow. 

\subsection{The Classical Gram's Law} For any integer $n$, \emph{the $n$-th Gram point $g_n$} is the unique solution of the equation
\[
\theta(g_n) = \pi n.
\]
These points take their name from J.P. Gram, who introduced a compelling observation known as Gram's law in his work \cite{G}. According to Gram's law, the Hardy Z-function $Z(t)$ at a Gram point $g_n$ generally satisfies the inequality

\begin{equation}
\label{eq:gram}
(-1)^{n} Z(g_n) > 0.
\end{equation}

Consequently, a Gram point that upholds this inequality is termed \emph{good}, whereas a point that fails to meet this condition is labeled \emph{bad}.

The importance of Gram's law arises from the fact that because $Z(t)$ is a real function, whenever two consecutive Gram points, say $g_n$ and $g_{n+1}$, are good, a zero of $Z(t)$ is guaranteed to exist between them. This observation would have substantial consequences for the RH if Gram's law were to hold for all integers, indeed, it would be a proof. However, the existence of bad Gram points, which violate Gram's law, introduces complexities into this potentially elegant correspondence.

Indeed, the earliest known violation of Gram's law occurs at \( n=126 \) and was discovered by Hutchinson's findings \cite{H}. The coexistence of Gram's law's overwhelming statistical success and its periodic exceptions brings forth the question: Is Gram's law an intrinsic property of \( Z(t) \) or merely an observational regularity? Edwards for instance, in \cite{E}, reflects on Gram's law as an initial, perhaps unsophisticated, attempt to predict the oscillatory behaviour of \( Z(t) \). He mentions that, surprisingly, it turned out to be much more successful than what might have been initially anticipated.

\subsection{The Discriminant and the Corrected Gram's Law} 
The classical discriminant of a quadratic function \(F(z ; \overline{a}) = a_2 z^2 + a_1 z + a_0\) is given by \(\Delta(\overline{a}) = a_1^2 - 4a_0 a_2\), which is itself a quadratic expression in the parameters \(\overline{a}\). The extremal point of \(F(z ; \overline{a})\) is \(g(\overline{a}) = -\frac{a_1}{2a_2}\) and hence 
\[
F(g(\overline{a}) ; \overline{a}) = \frac{a_1^2}{4a_2} - \frac{a_1^2}{2a_2} + a_0 = -\frac{a_1^2}{4a_2} + a_0 = 0 \Leftrightarrow \Delta(\overline{a}) = a_1^2 - 4a_0 a_2 = 0.
\] In our specific context, the discriminant \(\Delta_n(\overline{a} )\) is analogously defined. Locally, around zero we can extend $g_n(\overline{a})$ to be the variation of the Gram point $g_n$ with respect to $\overline{a}$, by defining it to be the corresponding 
extremal point of $Z_N(t ; \overline{a})$. The following is our main object of study: 

\begin{dfn}[$n$-th Gram discriminant] For any $n \in \mathbb{Z}$ we refer to \[ \Delta_n (\overline{a}):= Z_{N(n)}(g_n (\overline{a}) ; \overline{a}) \] with $N(n):= \left [ \frac{g_n}{2} \right ]$, as the \emph{$n$-th Gram discriminant of $Z(t)$}.  
\end{dfn} 

In particular, much like the quadratic discriminant serves as a measure for whether the two zeros of \(F(z; \overline{a})\) are real, our discriminant \(\Delta_n(r)\) can be conceptualized as a measure for the realness of the two consecutive zeros \(t_n(\overline{a})\) and \(t_{n+1}(\overline{a})\) of \(Z(t ; \overline{a})\). Our first main theorem is the following: 
\begin{mainthm}[Corrected Gram's law equivalent to RH] \label{mainthm:B} For any $n \in \mathbb{Z}$, The Riemann hypothesis holds if and only if the following corrected Gram's law holds $$(-1)^n \Delta_n ( 1,...,1) >0.$$  In particular, the extended Gram point $g_n(\overline{a})$ can be analytically continued to $\overline{1}=(1,...,1)$.  
\end{mainthm}
 
Contrary to the algebraic quadratic case, our discriminant $\Delta_n(\overline{a})$ as well as the extended Gram point $g_n(\overline{a})$ are transcendental functions and obtaining closed-form expressions for them is not feasible. We prove:

\begin{mainthm}[Second-order approximation of $\Delta_n(\overline{r})$] \label{thm:B1} For any $n \in \mathbb{Z}$ and $ r \in [0,1]$ the following second-order approximation holds $$ \Delta_n(\overline{r}) =Z(g_n;\overline{r})+ \frac{1}{2} H_n (0) \cdot r^2  + O(r^3), $$ where the second order Hessian at $r=0$ is given by   
\[ H_n(0) =2(-1)^n \left (  \frac{Z'(g_n)}{\ln \left ( \frac{g_n }{2 \pi } \right )}  \right )^2.
\]
Moreover, 
\[ 
Z'(g_n ) = \frac{1}{4} (-1)^n \ln^2 \left ( \frac{g_n}{2 \pi} \right ) \overline{1} \cdot \nabla g_n(0), 
\]
where 
\[
\nabla  g_{n} (\overline{0}) := \left (  \frac{\partial g_{n}}{\partial a_1} (\overline{0}) ,...,\frac{\partial g_{n}}{\partial a_N} (\overline{0}) \right ).
\]
is the gradient of $g_n(\overline{a})$ at $ \overline{a}=\overline{0}$. 
\end{mainthm} 

Theorem \ref{thm:B1} implies the following key result: 

\begin{cor} The following holds for any $n \in \mathbb{Z}$:
\begin{enumerate}

\item The classical Gram's law $(-1)^n Z(g_n)>0$ is the first-order approximation of the corrected Gram's law $(-1)^n \Delta_n( \overline{1})>0$. 
\item  The second-order Hessian $H_n(0)$ measures the magnitude of the local shift of $g_n(\overline{a})$ along the $t$-axis. 
\end{enumerate}

\end{cor} 

Combined, these formula reveal a deep property of the corrected Gram law: in order for the law to hold, the second-order term should function as a correcting term requiring a strong move of the position of $g_n(\overline{a})$ compensating for the violation of the classical law, for bad Gram points. This leads us to introduce the following notion:

\begin{dfn}[Viscosity] For any $n \in \mathbb{Z}$ we refer to the value of the logarithmic derivative of the \(Z\)-function
\[
\mu(g_n) = \frac{Z'(g_n)}{Z(g_n)},
\]
as the \emph{viscosity of the Gram point \(g_n\)}.
\end{dfn} 

The viscosity can be seen as a measurement for the relationship between the two aforementioned forces: The pull towards the axis expressed by $Z(g_n)$ and the shift along the axis expressed by $Z'(g_n)$. The above discriminant analysis leads us to discover a remarkable new empirical property of the $Z$-function:
\begin{conj}[Repulsion G-B-G conjecture] Assume $g_n$ is a bad Gram point with good consecutive neighbours $g_{n-1},g_{n+1}$. Then the following (non-sharp) viscosity bound holds
$$\abs{\mu(g_n)}>4.$$
\end{conj}

We argue that this newly discovered bound has profound implications regarding the behaviour of the zeros of the $Z$-function. Specifically, a repulsion phenomenon between consecutive zeros, which seems to be foundational for the validity of the Riemann Hypothesis itself. Based on the viscosity bound we suggest a general optimization approach for the establishment of the corrected Gram's law. 

Finally, recall that the Davenport-Helibronn function $\mathcal{D}(s)$ serves as a compelling counterexample to the Riemann Hypothesis, satisfying the necessary functional equation while violating the RH, in the sense that its zeros do not all lie on the critical line, see \cite{DH}. The elusive nature of \(\mathcal{D}(s)\) often underscores the challenge of identifying the unique properties that compel the $Z$-function to adhere to the RH, lacking in $\mathcal{D}(s)$. We show that, unlike the $Z$-function, where the violation of the Gram's law is attributed to the non-linearity of the discriminant, for \(\mathcal{D}(s)\), the failure of Gram's law is a genuine violation of the corrected Gram's law itself. In particular, we show that the Davenport-Heilbronn function does not satisfy a repulsion bound similar to the one observed for the $Z$-function. 

\bigskip

The rest of the work is organized as follows: Section \ref{s:2} recalls the computation of \(Z(t)\) via the approximate formula and defines the sections and core of \(Z(t)\). In Section \ref{s:3}, the global discriminant is defined and the \(A\)-philosophy is described. Section \ref{s:4} introduces the local discriminant \(\Delta_n(\overline{a})\) of two consecutive zeros and proves Theorem \ref{mainthm:B}.

Section \ref{s:5} studies the specific case of the linear curve and presents various examples, illustrating results proved in later sections. In Section \ref{s:6}, the first-order approximation of \(\Delta_n(\overline{r})\) is proved. Section \ref{s:7} computes the Hessian \(H_n(0)\), establishing the second-order approximation of \(\Delta_n(\overline{r})\) and concludes the proof of Theorem \ref{thm:B1}.

In Section \ref{s:8}, the viscosity of a Gram point is introduced, the experimental discovery of the G-B-G repulsion relation is described, and its possible relation to the Montgomery pair correlation conjecture is discussed. Section \ref{s:9} studies the violations of the corrected Gram's law and repulsion relation for the Davenport-Heilbronn function. Section \ref{s:10} presents an in-depth study of the repulsion relation by introducing the adjustments \(Z^{\pm}_c(g_n)\) and \(Z^{\pm}_s(g_n)\).

Section \ref{s:11} discusses the failure of the classic Newton method for the establishment of the RH, following Edwards. Section \ref{s:12} suggests a more refined non-linear optimization method, taking into account the geometry of the discriminant in \(A\)-space, for the establishment of the corrected Gram's law. Section \ref{s:13} presents a summary and concluding remarks.

\section{The $N$-Sections and Core of $Z(t)$} \label{s:2} 
The \( Z \)-function is formally defined as
\[
Z(t) = e^{i \theta(t)} \zeta \left( \frac{1}{2} + it \right),
\]
However, this definition is not conducive to practical calculations, which require the use of approximate formula. The Riemann-Siegel main sum, which is derived from the approximate functional equation (AFE), is a widely used alternative \cite{E}. In this study, we use a variant of the approximate functional equation that includes a greater number of terms, which is found to be better suited for the objectives of this study, see the following Remark \ref{rem:2.1}. This approximation is given as follows:
\begin{equation}
\label{eq:Z-function}
Z(t) = \cos(\theta(t))+ \sum_{k=1}^{N} \frac{1}{\sqrt{k+1} } \cos ( \theta (t) - \ln(k+1) t)+ O \left ( \frac{1}{t} \right ),
\end{equation}
where $N=\left [ \frac{t}{2} \right ]$, see \cite{SP1,SP2}. This leads to define:

\begin{dfnnonum}[$N$-th Section and Core of $Z$]
For any $N \in \mathbb{N}$, we denote
\[
Z_N(t) = \cos(\theta(t))+ \sum_{k=1}^{N} \frac{1}{\sqrt{k+1} } \cos ( \theta (t) - \ln(k+1) t),
\]
as the \emph{$N$-th section of $Z(t)$}. Specifically, we define
$Z_0(t) := \cos(\theta(t))$ as \emph{the core function of $Z(t)$}.
\end{dfnnonum}

Figure \ref{fig:f0} presents a comparison between $\ln \left | Z(t) \right | $ (in blue) and the core $\ln \left | Z_0(t) \right | $ (in orange) in the range $0 \leq t \leq 50$:

\begin{figure}[ht!]
\centering
\includegraphics[scale=0.5]{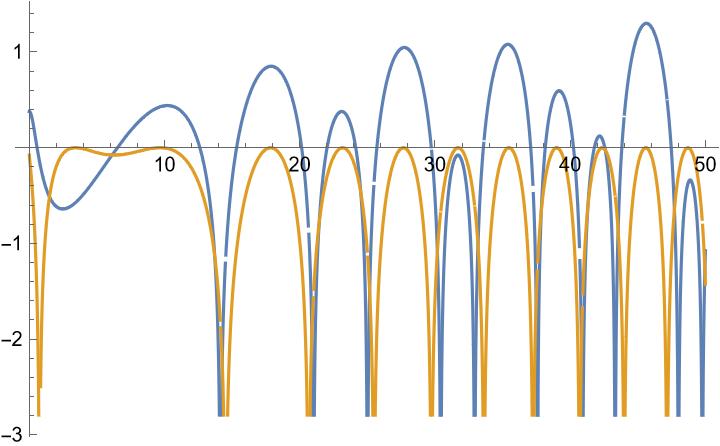}
\caption{\small{$\ln \left | Z(t) \right |$ (blue) and $\ln \left | Z_0 (t) ) \right | $ (orange) for $0 \leq t \leq 50$.}}
\label{fig:f0}
\end{figure}

The fact that the zeros of the core $Z_0(t)$ can be considered as rough approximations of the zeros of $Z(t)$ was actually observed by various authors, see for instance \cite{E,FL,J,SP1}, and might have already been known to Riemann himself, see Section \ref{s:11}. Similarly, the Gram points $g_n$, which are the extremal points of the core $Z_0(t)$, can be considered as rough approximations of the extremal points of $Z(t)$. 

 Recall that the Lambert $W$-function, denoted as $W_0(x)$, is a multi-valued function defined as the inverse of the function $ W_0(x) e^{W_0(x)}$, see \cite{CGHJK,Hay}. 
It is known that for any \( n \in \mathbb{Z} \), the $n$-th zero and extremal point of \( Z_0(t) \) are given respectively by 
 \[ \label{eq:tn}
\begin{aligned}
    t^0_n &= \frac{(8 n - 11) \pi}{4 \cdot W_0 \left ( \frac{8 n - 11}{8 \cdot e}\right )} & ; && 
    g_n &= \frac{(8 n + 1) \pi}{4 \cdot W_0 \left ( \frac{8 n + 1}{8 \cdot e}\right )},
\end{aligned}
\] see for instance \cite{FL} for the zeros and \cite{I} for the Gram points. The following two facts regarding the core $Z_0(t)$ which show that RH and Gram's law should actually be viewed as natural properties of the core: 
\begin{prop}[RH and Gram's law for $Z_0(t)$]
For any \( n \in \mathbb{Z} \):
\begin{description}
    \item[RH] The zero \( t^0_n \) of \( Z_0(t) \) is real. 
    \item[Gram's law] \( (-1)^n Z_0(g_n)>0 \).
\end{description}
\end{prop}

\begin{proof} The RH for the core $Z_0(t)$ follows directly from \eqref{eq:tn}. The Gram's law for the core follows from 
\[ (-1)^n Z_0(g_n) = (-1)^n cos ( \theta(g_n) ) = (-1)^n cos(\pi n) = (-1)^{2n} = 1 >0. \]
\end{proof} 

Consequently, the study of the Riemann Hypothesis and Gram's law can be re-framed as a question regarding the extent of deviation of $Z(t)$ from its core $Z_0(t)$, and from the fundamental properties of its zeros (RH) and extremal points (Gram's law). The following remark is due:
\begin{rem} \label{rem:2.1} It's important to note that the approximate formula for the $Z$-function \eqref{eq:Z-function} used here involves a summation of terms up to $[\frac{t}{2}]$. This follows from the simple approximate functional equation for $Z(t)$. More commonly in literature, the formula
\begin{equation}
\label{eq:Z-Hardy} Z(t) \approx 2 \sum_{n=0}^{\left [ \sqrt{\frac{t}{2 \pi} } \right ]} \frac{1}{\sqrt{n+1} } \cos ( \theta (t) - \ln(n+1) t),
\end{equation}
justified by the Hardy-Littlewood approximate functional equation, is used, where the summation is taken up to $\left [ \sqrt{\frac{t}{ 2\pi }} \right ]$, see \cite{I}. The primary advantage of this latter formula is its efficiency, as it requires the computation of far fewer terms relative to $t$. However, its main drawback is that it is not sufficiently sensitive to discern the RH by itself, as it exhibits non-real zeros and necessitates further development of the error term via the Riemann-Siegel formula. In contrast, the more robust formula \eqref{eq:Z-function}, although more computationally intensive, has been observed to be sufficient for discerning the RH, as initially noted by Spira in his empirical investigations, see \cite{SP1,SP2}, see also Remark \ref{rem:crucial}. 
\end{rem}

\section{The \( A \)-Philosophy and Discriminant} \label{s:3}
Originating from the field of algebraic geometry, the \( A \)-philosophy, as introduced by Gelfand, Kapranov, and Zelevinsky in \cite{GKZ}, advocates for studying mathematical objects not in isolation but rather in relation to a broader parameter space. This approach becomes particularly powerful when analyzing the discriminant hypersurface that forms within this parameter space, often revealing essential insights about the original mathematical object. While the \( A \)-philosophy has primarily found applications in algebraic contexts, we aim to extend it to the transcendental setting of the Z-function. This extension opens up new avenues for inquiry, significantly enriching our understanding of how such functions behave as the parameters of their approximating sums vary.

\begin{dfnnonum}[\( N \)-th Parameter Space]
For a given \( N \in \mathbb{N} \), we define \( \mathcal{Z}_N \) to be the parameter space consisting of functions of the form
\[
Z_N(t; \overline{a}) = Z_0(t) + \sum_{k=1}^{N} \frac{a_k}{\sqrt{k+1}} \cos (\theta(t) - \ln(k+1) t),
\]
where \( \overline{a} = (a_1, \ldots, a_N) \) belongs to \( \mathbb{R}^N \). 
\end{dfnnonum}

This new space \( \mathcal{Z}_N \) allows us to investigate how subtle changes in \( \overline{a} \) influence the behaviour of \( Z_N(t; \overline{a}) \), which is crucial for our later discussions. Our study focuses on the following concept of the global discriminant within the \( \mathcal{Z}_N \) space:
\begin{dfnnonum}[Global Discriminant]
We define the global discriminant hyper-surface \( \Sigma_N \) in \( \mathcal{Z}_N \) as:
\[
\Sigma_N = \left \{ \overline{a} \in \mathbb{R}^N \mid Z_N(t; \overline{a}) \textrm{ has a multiple zero} \right \},
\]
where by multiple zero we mean parameters for which the function \( Z_N(t; \overline{a}) \) and its first derivative both have a common zero.
\end{dfnnonum}

Our first theorem establishes a geometric connection between the RH and our discriminant study:

\begin{theorem} \label{thm:A}
Consider \( \gamma(r) \) as a parametrized curve in \( \mathcal{Z}_N \) for \( r \in [0, 1] \), originating from the core function \( Z_0(t) \) at \( r = 0 \). 
\begin{enumerate}
\item The zero set of \( Z_N(t; \gamma(r)) \) is self-conjugate for any \( r \in [0, 1] \). That is, if \( z \) is a zero of \( Z_N(t; \gamma(r)) \), then its complex conjugate \( \overline{z} \) is also a zero.
\item The real zeros \( t_n(\overline{r}) \) of \( Z_N(t; \gamma(r)) \) remain well-defined, smooth, and real as long as they do not collide with other consecutive zeros.
\end{enumerate}
\end{theorem}
\begin{proof}
Note that since $\overline{\theta(\overline{t})}=\theta(t)$, all the terms of $Z_N(t; \gamma(r))$ satisfy 
\[
\overline{\cos(\theta(\overline{t})-\ln(k+1)\overline{t})}=\cos(\theta(t)-\ln(k+1)t),
\]
from which the self-conjugacy of (1) results, the rest follows naturally. 
\end{proof} 

Following the proof of Theorem \ref{thm:A}, we infer that the zeros of the function can only exit the real line as pairs at points of collision, where multiple zeros appear. This notion allows us to reframe the Riemann Hypothesis (RH) in a novel way:
\begin{conj}[$A$-Philosophy dynamic RH] \label{con:1} 
For any $n \in \mathbb{N}$, there exists a path $\gamma(r)$ in $\mathcal{Z}_{\left [ \frac{t_n}{2} \right ]}$ with $\gamma(0)=Z_0(t)$ and $\gamma(1)=Z_N (t)$ which is non-colliding for the consecutive pair of zeros $t_n(r)$ and $t_{n+1}(r)$ for $r \in [0,1]$.
\end{conj}
This conjecture, if true, would imply a one-to-one, order-preserving correspondence between the zeros (or extremal points) of the core function \(Z_0(t)\) and those of \(Z(t)\). 
\section{The Corrected Gram's Law} \label{s:4}
The study of elements in \(\mathcal{Z}_N\) is particularly challenging due to their infinite number of zeros. This complexity is compounded when considering the geometry of the global discriminant \(\Sigma_N \subset \mathcal{Z}_N\). To simplify our discussion and gain more precise insights, we introduce the concept of a \textit{local discriminant} for a given pair of consecutive zeros \(t_n\) and \(t_{n+1}\) and a given \(n \in \mathbb{Z}\), with respect to a one-parametric family in \(\mathcal{Z}_N\). To offer a concrete example, Fig. \ref{fig:f2} illustrates the collision process between the $16$-th and $17$-th zeros of the $1$-parametric family
\[
Z_1(t; r) := \cos(\theta(t)) + \frac{r}{\sqrt{2}} \cos( \theta(t) - \ln(2) t) \in \mathcal{Z}_1,
\]
for \(t\) in the range \(66.5 \leq t \leq 70\) and for \(r=0\) (blue), \(r=0.75\) (orange), \(r=1.5\) (green), and \(r=2.25\) (red), showing  how the zeros evolve with different values of \(r\):

\begin{figure}[ht!]
	\centering
		\includegraphics[scale=0.4]{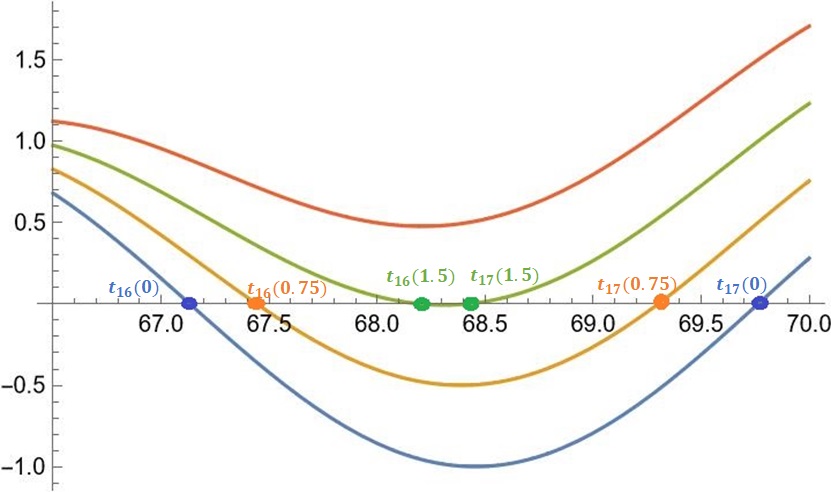}
	\caption{\small{$Z_1(t; a)$ in the range $66.5 \leq t \leq 70$ for $r=0$ (blue), $r=0.75$ (orange), $r=1.5$ (green) and $r=2.25$ (red).}}
\label{fig:f2}
	\end{figure} 
Note that a multiple zero occurs precisely when \( g_n(r) \) is itself a zero of \( Z_N(t; \gamma(r)) \). This phenomenon, evident in Fig. \ref{fig:f2}, is significant because it characterizes points where the function loses its simple zero-crossing behavior. We say that \(\gamma\) is \emph{non-degenerate} for the \(n\)-th Gram point if \( g_n(r; \gamma) \) is well-defined, real, and varies continuously for any \( r \in [0,1] \). 

\begin{dfnnonum}[The \( n \)-th Gram discriminant of a non-degenerate curve]
Let \( \gamma \) be a non-degenerate curve for the \( n \)-th pair of zeros. We refer to
\[
\Delta_n(r; \gamma) := Z_N(g_n(r; \gamma); \gamma(r))
\]
as the \emph{\( n \)-th Gram discriminant of \( \gamma \)}.
\end{dfnnonum}

We now arrive at a central point of our investigation. 
\begin{theorem}[The corrected Gram law] \label{thm:B}
The Riemann hypothesis holds if and only if for any \( n \in \mathbb{Z} \) there exists a non-degenerate curve \( \gamma_n \) with \( \gamma(1) = (1, \ldots, 1) \) such that
\begin{equation}
\label{eq:Corrected-Gram-1}
(-1)^{n} \Delta_n(1; \gamma) > 0.
\end{equation}
\end{theorem}

\begin{proof}
By definition, for any curve \( \gamma \), starting at the origin \( \gamma(0) = (0, \ldots, 0) \), we have
\[
\Delta_n(0; \gamma) = Z_N(g_n(0); 0) = \cos(\theta(g_n)) = (-1)^{n}.
\]
Our proof demonstrates that the non-collision of \( t_n(r) \) and \( t_{n+1}(r) \) with respect to \( \gamma \) is tantamount to requiring that the discriminant \( \Delta_n(r; \gamma) \) remains invariant in sign.
\end{proof}

 Note that, contrary to the classical Gram law \eqref{eq:gram}, which is an empirical observation regarding the numerical tendency of Gram points, the corrected Gram's law \eqref{eq:Corrected-Gram-1} is expected to hold for all Gram points and is equivalent to the RH.
 At this point enters the second fundamental feature of the $A$-philosophy (aside from the existence of discriminants), which is the ability to study smooth variations through derivatives. In the next sections we will describe results regarding the geometrical content of the first and second derivatives.  
 
\section{The Discriminant of the Linear Curve - First Examples} \label{s:5} 
In this section, we examine the \(1\)-parametric family of functions defined by
\[
Z_N(t;r):= Z_0(t) + r \cdot \sum_{k=1}^{N} \frac{1}{\sqrt{k+1}} \cos (\theta (t) - \ln(k+1) t) \in \mathcal{Z}_N,
\]
for \(r \in [0,1]\). Note that $Z_N(t;r)$ is the curve in the parameter $A$-space, interpolating between the core function \(Z_0(t)\) and \(Z_N(t;\overline{1})\) by gradually adding all the terms together in a proportional manner. Further motivation for considering $Z_N(t; r)$ will be detailed in Section \ref{s:11}, where we discuss Edwards' speculation. To illustrate how $Z_N(t;r)$ behaves with respect to good$\setminus$bad Gram points, let us examine the following example:
\begin{ex}[$\Delta_n(r)$ for $g_{90}$ and $g_{126}$] \label{ex:1} Figure \ref{fig:f3} shows $\Delta_n (r)$ (blue) and its first order approximation $Z_N (g_n ; r)$ (orange) for the good Gram point $n=90$ (left) and the bad Gram point $n=126$ (right), with $0 \leq r \leq 1$:

\begin{figure}[ht!]
	\centering
		\includegraphics[scale=0.4]{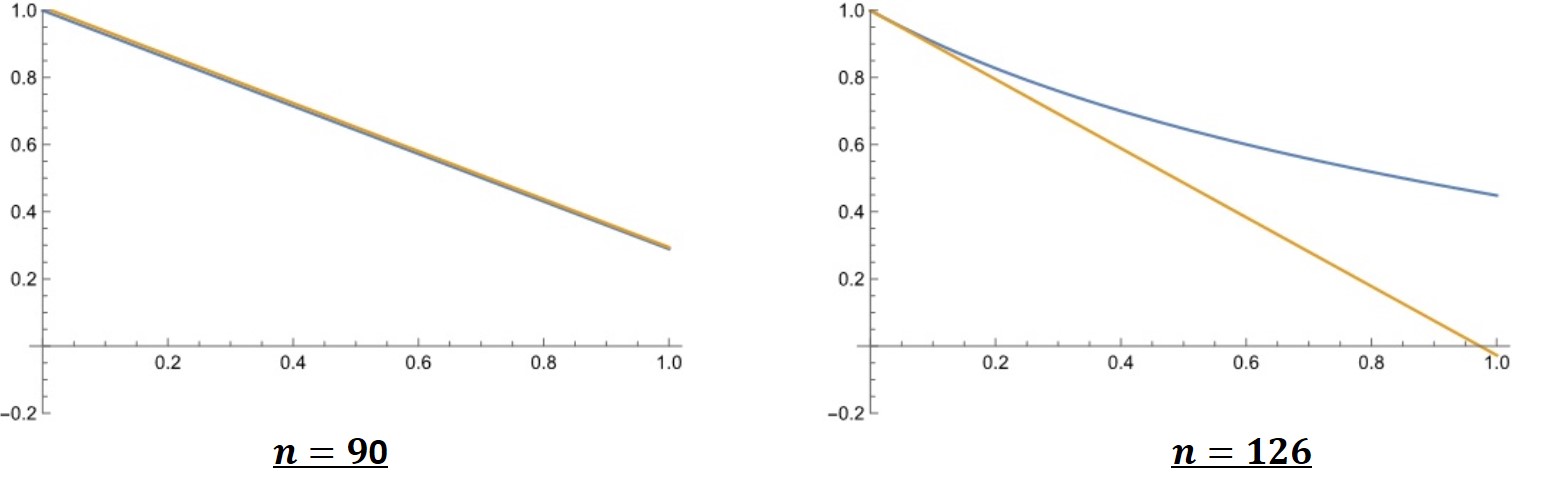} 	
		\caption{\small{Graph of $\Delta_n (r)$ (blue) and $Z_N (g_n ; r)$ (orange) for $n=90$ (left) and $n=126$ (right) with $0 \leq r \leq 1$.}}
\label{fig:f3}
	\end{figure} 
	
For the good Gram point $g_{90}$ the first-order approximation $Z_N (g_n ;r)$ serves as a rather accurate approximation of $\Delta_n (r)$ itself. For the bad Gram point $g_{126}$ the first-order approximation $Z_N (g_n ; r)$ is seen to deviate from $\Delta_n (r)$. However, for both $g_{90}$ and $g_{126}$ the linear family $Z_N(t;r)$ is non-colliding and establishes the corrected Gram's law. Consider Fig. \ref{fig:f4} which shows the graphs of $Z_N(t;r)$ themselves in the range $t \in [g_n -2,g_n+2]$ for $n=90$ (left) and $n=126$ (right) and various values of $r \in [0,1]$.
\begin{figure}[ht!]
	\centering
		\includegraphics[scale=0.35]{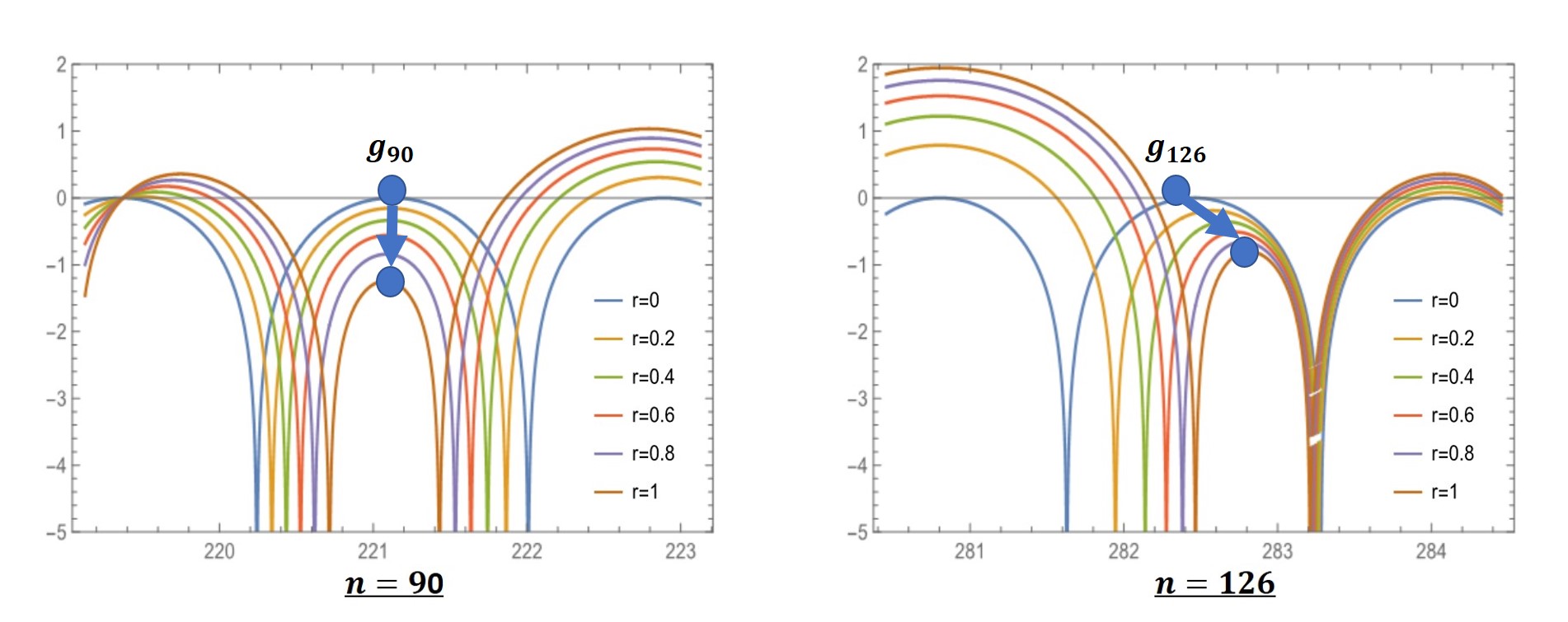} 	
		\caption{\small{Graphs of $Z_N(t;r)$ in the range $t \in [g_n -2,g_n+2]$ for $n=90$ (left) and $n=126$ (right) and various values of $r \in [0,1]$.}}
\label{fig:f4}
	\end{figure}

\end{ex}

 Our aim in the subsequent sections is to give the theoretical explanation for the phenomena observed in the above examples.
 In Section \ref{s:6} we show that $Z_N(g_n ; r)$ is the first order-approximation of $\Delta_n(r)$, explaining the close proximity between 
 the two, observed for the good Gram point $n=91$. In Section \ref{s:7} we compute the second-order approximation, explaining the shift in $g_n(r)$ to the right, observed for the bad Gram point $n=126$. 
 
 \section{The First-Order Approximation of the Corrected Law is the Classical Law} \label{s:6}
Let us consider the first-order approximation of $\Delta_n (\overline{r})$ given by 
\[
\Delta_n(r) = \Delta_{n}(\overline{0})+ \nabla \Delta_{n}  (\overline{0}) \cdot \overline{r}+O(r^2),
\]
where the gradient vector is given by 
\[
\nabla  \Delta_{n} (\overline{0}) := \left (  \frac{\partial \Delta_{n}}{\partial a_1} (\overline{0}) ,...,\frac{\partial \Delta_{n}}{\partial a_N} (\overline{0}) \right ).
\] 
We have: 
 
\begin{theorem}[First order approximation is Gram's law] \label{thm:first-order} For any $n \in \mathbb{Z}$, the first-order approximation of the discriminant $\Delta_n(r)$ of the linear curve is given by  
\[ 
\Delta_{n}(r) = Z(g_n ; r) +O(r^2).
\] 
In particular, the classical Gram's law is the first-order approximation of the corrected Gram's law for the linear curve. 
\end{theorem} 

\begin{proof} 
Consider the function 
\[
F_k(t;a):= \cos(\theta(t)) + \frac{a}{\sqrt{k+1}} \cos(\theta(t)- \ln(k+1)t)
\] 
and set 
\[ 
G_k(t;a):= \frac{\partial}{\partial t} F_k(t; a). 
\]
Denote by $g_n(a)$ the extremal point of $F_k(t;a)$ locally extending the gram point $g_n$. Then, by definition, the discriminant can be written as 
\[ 
\Delta_n( 0,...,0,a,0,...,0) = F_k(g_n(a);a). 
\]
Hence, by the chain rule, we have 
\[ 
\frac{\partial \Delta_{n}}{\partial a_k} (\overline{0})=\frac{\partial}{\partial a} F_k(g_n(a) ; a)(0)= \frac{\partial g_n}{\partial a}(0) \cdot G_k(g_n ; 0)+ \frac{\partial}{\partial a} F_k (g_n ; 0). 
\]
But since the Gram points are exactly the solutions of $G_k(g_n;0)=0$, we get  
\[ 
\frac{\partial \Delta_{n}}{\partial a_k} (\overline{0})=\frac{\partial}{\partial a} F_k (g_n ; 0)=\frac{1}{\sqrt{k+1} } \cos ( \theta (g_n) - \ln(k+1) g_n).
\]
The first-order approximation of the discriminant $\Delta_n(r)$ is given by   
\begin{multline} 
\Delta_{n}(r) \approx \Delta_{n}(\overline{0})+ \nabla \Delta_{n}  (\overline{0})  \cdot \overline{r}= \\ = 
\Delta_{n}(\overline{0})+ r\cdot \sum_{k=1}^N \frac{\partial \Delta_{n}}{\partial a_k} (\overline{0})= \\ = \cos(\theta(g_n)) +\sum_{k=0}^N \frac{r}{\sqrt{k+1} } \cos ( \theta (g_n) - \ln(k+1) g_n) = Z(g_n ; r).
\end{multline} 
\end{proof} 

\begin{rem}[The reason for Gram's law] 
Theorem \ref{thm:first-order} can be seen as giving a theoretical explanation to the empirical phenomena of the classical Gram law, as following from the RH. Indeed, for good Gram points the first-order approximation $Z(g_n)$ is close to the value of $\Delta_n(1)$, and hence is expected to satisfy $(-1)^n Z(g_n)>0$.
 \end{rem}

 \section{The Second-Order Approximation of $\Delta_n(r)$} \label{s:7}
We can now consider the second-order approximation of $\Delta_n (\overline{r})$, which in view of Theorem \ref{thm:first-order} can be written as 
\[ 
\Delta_n(r) = Z(g_n ; r) +\frac{1}{2} H_n(\overline{0}) \cdot r^2+O(r^3),
\]
where 
\[
H_n(\overline{0}):=\sum_{k_1,k_2=1}^N \frac{\partial^2 \Delta_{n}}{\partial a_{k_1} \partial a_{k_2}} (\overline{0})
\]
is the Hessian of second derivatives.
The main result of this section shows the content of the second-order Hessian:
\begin{theorem}[Second-order approximation]  \label{thm:CM} For any $n \in \mathbb{Z}$, the second-order Hessian is given by 
\[ \label{eq:Hess} 
H_n:=2(-1)^n \left (  \frac{Z'(g_n)}{\ln \left ( \frac{g_n }{2 \pi } \right )}  \right )^2.
\]
\end{theorem}
In order to prove the main theorem let us first prove a few preliminary results: 
\begin{lemma} 
	\label{Lem:7.2}
	\[ 
\frac{\partial}{\partial a_k} g_n(\overline{a})=  \frac{\sin ( \theta (g_n(\overline{a})) - \ln(k+1) g_n(\overline{a}))}{2 \sqrt{k+1}  Z''(g_n(\overline{a}) ; \overline{a}) } \ln \left (\frac{g_n(\overline{a})}{2 \pi (k+1)^2} \right ).
\]
	In particular, 
	\[
\frac{\partial}{\partial a_k} g_n(\overline{0})=  2 (-1)^{n+1} \frac{\sin ( \theta (g_n) - \ln(k+1) g_n)}{ \sqrt{k+1}  \ln^2 \left ( \frac{g_n}{2 \pi} \right ) } \ln \left (\frac{g_n}{2 \pi (k+1)^2} \right ). 
\]
	\end{lemma} 

\begin{proof} 

Let $g_n(\overline{a}; \epsilon)$ 
be the $n$-th extremal point of 
\[
F_{k,\epsilon}(t;\overline{a}):= Z_N (t ; \overline{a}) + \frac{\epsilon}{\sqrt{k+1}} \cos(\theta(t)- \ln(k+1)t).
\] 
for $0<\epsilon$ small enough. That is, the zero of the equation 
\[
G_{k,\epsilon}(t;\overline{a}):= \frac{\partial}{\partial t} F_{k,\epsilon}(t; a)=0. 
\] 
Then according to Newton's method, one can take the following first iteration
\[
\widetilde{g}_n(\overline{a};\epsilon) := g_n (\overline{a}) - \frac{G_{k,\epsilon}(g_n(\overline{a}) ; \overline{a})}{G_{k,\epsilon}'(g_n(\overline{a}) ;\overline{a})}, 
\] 
as an approximation of $g_n(\overline{a};\epsilon)$, which improves as $\epsilon$ decreases, see \cite{SM}. Note that 
\[
G_{k,\epsilon}(t;\overline{a})=Z'_N (t ; \overline{a}) - \frac{\epsilon }{\sqrt{k+1}} \sin(\theta(t)- \ln(k+1)t)(\theta'(t)-\ln(k+1)). 
\]
Since $Z'_N (g_n (\overline{a}) ; \overline{a})=0$ and 
\[
\theta'(t) = \left ( \frac{t}{2} \ln \left ( \frac{t}{2 \pi} \right ) - \frac{t}{2} -\frac{\pi}{8} \right) ' = \frac{1}{2} \ln \left ( \frac{t}{2 \pi} \right ),
\]
we have 
\[
G_{k,\epsilon}(g_n(\overline{a}) ;\overline{a})= - \frac{\epsilon}{2\sqrt{k+1}} \sin(\theta(g_n (\overline{a}) )- \ln(k+1)g_n (\overline{a}) )\cdot  \ln \left ( \frac{g_n( \overline{a})}{2 \pi (k+1)^2} \right ). 
\]
For the derivative the main term is given by 
\[
G'_{k,\epsilon}(g_n ( \overline{a}) ;\overline{a})= Z''_N(g_n(\overline{a})  ; \overline{a})+ O(\epsilon).
\] 
In particular, we have 
\[
Z''_N(g_n;\overline{0})= -\cos(\theta(g_n)) (\theta'(g_n) )^2 =\frac{(-1)^{n+1}}{4} \ln^2 \left ( \frac{g_n}{2 \pi} \right ),
\]
as required.
\end{proof} 

We have: 

\begin{prop} \label{prop:7.2} For any $1 \leq k_1,k_2 \leq N$ the following holds: 
\[
\frac{\partial^2 \Delta_{n}}{\partial a_{k_1} \partial a_{k_2}} (\overline{a}) =-\frac{1}{4 Z''(g_n(\overline{a}) ; \overline{a}) } \prod_{i=1}^2 
\frac{\sin(\theta(g_n(\overline{a}))- \ln(k_i+1) g_n (\overline{a})) \cdot \ln \left ( \frac{g_n (\overline{a})}{2 \pi (k_i+1)^2 } \right )}{\sqrt{k_i+1} } .
\]
In particular,
\[
\frac{\partial^2 \Delta_{n}}{\partial a_{k_1} \partial a_{k_2}} (\overline{0}) =\frac{(-1)^n}{\ln^2 \left ( \frac{g_n }{2 \pi } \right )} \prod_{i=1}^2 
\frac{\sin(\ln(k_i+1) g_n) \cdot \ln \left ( \frac{g_n }{2 \pi (k_i+1)^2 } \right )}{\sqrt{(k_i+1)} } .
\] 
\end{prop} 

\begin{proof}
Consider the function 
\begin{multline} 
F_{k_1,k_2,\epsilon_1,\epsilon_2}(t;\overline{a}):= Z_N(t ; \overline{a}) + \frac{\epsilon_1}{\sqrt{k_1+1}} \cos(\theta(t)- \ln(k_1+1)t)+ \\ +\frac{\epsilon_2}{\sqrt{k_2+1}} \cos(\theta(t)- \ln(k_2+1)t)
\end{multline} 
and set 
\begin{equation} 
G_{k_1,k_2,\epsilon_1,\epsilon_2}(t;\overline{a}):= \frac{\partial}{\partial t} F_{k_1,k_2,\epsilon_1,\epsilon_2}(t; \overline{a}). 
\end{equation} 
Denote by $g_n(\overline{a}; \epsilon_1,\epsilon_2)$ the extremal point of $F_{k_1,k_2}(t;a_1,a_2)$ locally extending the gram point $g_n(\overline{a})$. Then, by definition, the discriminant can be written as 
\[
\Delta_n( a_1,...,a_{k_1}+\epsilon_1,...,a_{k_2}+\epsilon_2,..,a_N) = F_{k_1,k_2,\epsilon_1,\epsilon_2}(g_n(\overline{a}; \epsilon_1, \epsilon_2); \overline{a} ). 
\] 
Hence, by the chain rule, we have 
\begin{multline} 
\frac{\partial \Delta_{n}}{\partial \epsilon_1} ( a_1,...,a_{k_1}+\epsilon_1,...,a_{k_2}+\epsilon_2,..,a_N) =\frac{\partial}{\partial \epsilon_1} F_{k_1,k_2,\epsilon_1,\epsilon_2}(g_n(\overline{a}; \epsilon_1, \epsilon_2); \overline{a} )=\\= 
\frac{\partial g_n}{\partial \epsilon_1}(\overline{a}; \epsilon_1,\epsilon_2) \cdot G_{k_1,k_2,\epsilon_1,\epsilon_2}(g_n(\overline{a};\epsilon_1,\epsilon_2) ; \overline{a})+ \frac{\partial}{\partial \epsilon_1} F_{k_1,k_2,\epsilon_1,\epsilon_2} (g_n(\overline{a}; \epsilon_1,\epsilon_2) ; \overline{a})=\\=\frac{\partial}{\partial \epsilon_1} F_{k_1,k_2,\epsilon_1,\epsilon_2} (g_n(\overline{a}; \epsilon_1,\epsilon_2) ; \overline{a})=\\ 
=\frac{1}{\sqrt{k_1+1}} \cos(\theta(g_n(\overline{a};\epsilon_1,\epsilon_2))- \ln(k_1+1) g_n(\overline{a}; \epsilon_1,\epsilon_2)). 
\end{multline} 
Again, we use the fact that $g_n(\overline{a}; \epsilon_1,\epsilon_2)$ are, by definition, the solutions of 
\[
G_{k_1,k_2,\epsilon_1,\epsilon_2}(g_n(\overline{a}; \epsilon_1,\epsilon_2);\overline{a})=0.
\] 
Thus, for the second derivative we have 
\begin{multline} 
\frac{\partial^2 \Delta_{n}}{\partial \epsilon_1 \partial \epsilon_2} ( a_1,...,a_{k_1}+\epsilon_1,...,a_{k_2}+\epsilon_2,..,a_N)   =\\ 
=\frac{\partial}{\partial \epsilon_2} \left ( \frac{1}{\sqrt{k_1+1}} \cos(\theta(g_n(\overline{a};\epsilon_1,\epsilon_2))- \ln(k_1+1) g_n(\overline{a}; \epsilon_1,\epsilon_2)) \right )= \\ 
= -\frac{1}{2\sqrt{k_1+1}} \sin(\theta(g_n(\overline{a};\epsilon_1,\epsilon_2))- \ln(k_1+1) g_n(\overline{a}; \epsilon_1,\epsilon_2)) \cdot \ln \left ( \frac{g_n(\overline{a}; \epsilon_1,\epsilon_2)}{2 \pi (k+1)^2 } \right )  \frac{\partial g_n}{\partial \epsilon_2} (\overline{a} ; \epsilon_1,\epsilon_2). 
\end{multline} 
By substituting $(\epsilon_1,\epsilon_2)=(0,0)$ and applying Lemma \ref{Lem:7.2} the result follows.   
\end{proof} 
\begin{proof} [Proof of Theorem 7.1:] By Proposition \ref{prop:7.2} we have 
\begin{multline}
H_n(\overline{a}):=\sum_{k_1,k_2=1}^N \frac{\partial^2 \Delta_{n}}{\partial a_{k_1} \partial a_{k_2}} (\overline{0}) a_{k_1} a_{k_2},
=\\= \frac{(-1)^n}{\ln^2 \left ( \frac{g_n }{2 \pi } \right )} \sum_{k_1,k_2=1}^N  \prod_{i=1}^2 
\frac{\sin(\ln(k_i+1) g_n) \cdot \ln \left ( \frac{g_n }{2 \pi (k_i+1)^2 } \right )}{\sqrt{k_i+1} } a_{k_i}=\\=
\frac{(-1)^n}{\ln^2 \left ( \frac{g_n }{2 \pi } \right )} \left ( \sum_{k=1}^N  
\frac{\sin(\ln(k+1) g_n) \cdot \ln \left ( \frac{g_n }{2 \pi (k+1)^2 } \right )}{\sqrt{k+1} } a_{k}\right )^2=
4(-1)^n \left (  \frac{Z'(g_n; \overline{a})}{\ln \left ( \frac{g_n }{2 \pi } \right )}  \right )^2,
\end{multline} 
as required. 
\end{proof} 

Theorem \ref{thm:CM} shows that the magnitude of the Hessian $H_n$ is determined by the size of $Z'(g_n)$. We further have: 
\begin{theorem} \label{thm:Grad}
Moreover, 
\[ 
Z'(g_n ; \overline{r} ) = \frac{1}{4} (-1)^n \ln^2 \left ( \frac{g_n}{2 \pi} \right ) \overline{r} \cdot \nabla g_n(0), 
\]
where 
\[
\nabla  g_{n} (\overline{0}) := \left (  \frac{\partial g_{n}}{\partial a_1} (\overline{0}) ,...,\frac{\partial g_{n}}{\partial a_N} (\overline{0}) \right )
\]
is the gradient of $g_n(\overline{a})$ at $ \overline{a}=\overline{0}$.
\end{theorem} 
\begin{proof} 
The following holds 
\begin{multline} 
Z'_N(t; \overline{a}) = -\sin(\theta(t))\theta'(t) -\sum_{k=1}^{N} \frac{a_k}{\sqrt{k+1} } \sin ( \theta (t) - \ln(k+1) t) (\theta'(t)-\ln(k+1)) = 
\\=  -\frac{1}{2} \sin(\theta(t))\ln \left ( \frac{t}{2 \pi} \right )  -\sum_{k=1}^{N} \frac{a_k}{2\sqrt{k+1} } \sin ( \theta (t) - \ln(k+1) t)  \ln \left ( \frac{t}{2 \pi (k+1)^2} \right ). 
\end{multline} 
Hence, 
\[
Z'_N(g_n ; \overline{a})= \sum_{k=1}^{N} \frac{a_k}{2\sqrt{k+1} } \sin (\ln(k+1) g_n)  \ln \left ( \frac{g_n}{2 \pi (k+1)^2} \right ). 
\] 
But also, 
\[
\frac{1}{4} \ln^2 \left ( \frac{g_n}{2 \pi} \right ) \frac{\partial}{\partial a_k} g_n(\overline{0})=   (-1)^{n} \frac{\sin ( \ln(k+1) g_n)}{ 2 \sqrt{k+1} } \ln \left (\frac{g_n}{2 \pi (k+1)^2} \right ). 
\] 
Hence, 
\[
Z'(g_n ; \overline{a} ) = \frac{1}{4} (-1)^n \ln^2 \left ( \frac{g_n}{2 \pi} \right ) \overline{a} \cdot \nabla g_n(0),
\]
as required. 
\end{proof} 
Theorem \ref{thm:CM} and Theorem \ref{thm:Grad} together imply the following: 

\begin{cor} \label{Hess-grad} The following holds:  
\begin{enumerate} 
\item  The second-order Hessian $H_n$ is a measurement of the magnitude of the gradient $\nabla g_n(0)$. 
\item The direction of the shift of the $n$-th extremal point of $Z(t; \overline{a})$ with respect to $g_n$ is given by the sign of 
$(-1)^n Z'(g_n ; \overline{a})$.   
\end{enumerate}
\end{cor} 

The corrected Gram's law implies that for bad Gram points the second order term, represented by the Hessian $H_n$, is expected to become crucial in order to compensate on the first-order violation of the classical law. Corollary \ref{Hess-grad} further shows that a large second-order Hessian $H_n$ is expressed by the fact that the Gram point must experience a considerable positional shift for the approximation to be valid. Conversely, a small Hessian at a Gram point signifies that the first order term predominates in determining the position of the Gram point. For instance, consider: 

\begin{ex}[The Hessians of $g_{90}$ and $g_{126}$]
By direct computation, the Hessians of $\Delta_n(r)$ for the good Gram point $g_{90}$ and the bad Gram point $g_{126}$ are given by: 
\[ 
\begin{array}{ccc} H_{90} =0.00203615 & ; & H_{126} = 2.22893 \end{array} 
\]  
	In accordance with Example \ref{ex:1}, the Hessian $H_{90}$ is relatively small which implies that for the good Gram point $g_{90}$ the discriminant is largely determined by its first-order approximation $Z(g_n;r)$ and $g_n(r)$ hardly changes position. In contrast, the Hessian for the bad Gram point $g_{126}$ is considerably large. This conveys that $g_{126}(r)$ undergoes a considerable displacement from its original position $g_{126}$ as $r$ increases. This practical shift is visually demonstrated as a shift to the right in Fig. \ref{fig:f4}.
\end{ex}

 \begin{rem} Adding higher-order terms improves the level of accuracy of the approximation near $\overline{a}=\overline{0}$. However, the level of accuracy around $\overline{a}=\overline{1}$, can actually decrease. In particular, obtaining a reasonable approximation of $\Delta_n ( \overline{1})$ via Taylor approximation might require a rather substantial amount of higher degree terms, which would be completely infeasible to compute in practice, for general $n \in \mathbb{Z}$. 
\end{rem} 
 
 \section{An Experimental Investigation of the Viscosity of Gram Points}\label{s:8} 

In this section, drawing inspiration from fluid dynamics, we introduce the concept of 'viscosity' to quantify the positional shifting behaviour and conduct an initial empirical investigation of its properties. Generally, Corollary \ref{Hess-grad} straightforwardly suggests that if $g_n$ is a bad Gram point, one should anticipate a significant shift in the position of $g_n(r)$ itself to fulfil the corrected Gram's law at the second-order level. In other words, bad Gram points $g_n$ are expected to demonstrate some correlation between the values of $Z(g_n)$ and $Z'(g_n)$. For a bad Gram point, we interpret $Z(g_n)$ as pushing the extremal point downward towards zero to create a collision, while $Z'(g_n)$ pushes the extremal point sideways to avoid a collision. Thus, let us introduce the following definition:
 
 \begin{dfnnonum}[Viscosity of a Gram point] For any $n \in \mathbb{Z}$ we refer to
 of the $Z$-function 
\[
\mu(g_n) := \left \vert \frac{Z'(g_n)}{Z(g_n)} \right \vert
\]
as \emph{the viscosity of the Gram point $g_n$.}
 \end{dfnnonum}
 
 The viscosity of a Gram point $\mu(g_n)$ is essentially a measure of how much a Gram point 'resists' a change in its position, analogous to how viscosity in a fluid quantifies its resistance to flow. In essence, a high viscosity at a Gram point implies a lower tendency for the point to maintain its position and vice versa. 
 
 \begin{rem}
 The occurrence of the logarithmic derivative $\frac{Z'(t)}{Z(t)}$ in the study of the zeros of $Z(t)$ is not surprising, given its integral role as manifested by the classical formula:
\[
\frac{Z'(t)}{Z(t)}=i \theta'(t) -\frac{1}{it-\frac{1}{2}}+\sum_{\rho} \frac{1}{\frac{1}{2}+it-\rho}- \frac{1}{2} \frac{\Gamma'(\frac{5}{4}+\frac{it}{2})}{\Gamma(\frac{5}{4}+\frac{it}{2})}+\frac{1}{2}\ln(\pi),
\]
This relation, outlined in Section 3.2 of \cite{E}, underscores the inherent connection between the logarithmic derivative of $Z(t)$ and its zeros, further motivating our current study. However, in our case our interest is focused on the values of the logarithmic derivative at the bad Gram points $g_n$.
\end{rem} 

From our second-order approximation of the discriminant $\Delta_n(r)$, we thus come to anticipate that $\mu(g_n)$ of bad Gram points will exhibit different unique features compared to those of general general Gram points. Indeed, consider Fig. \ref{fig:f51} which presents the outcome of a numerical computation of the viscosity \(\mu(g_n)\) for the first \(n=1,\ldots,1000\) general Gram points. 
\begin{figure}[ht!]
	\centering
		\includegraphics[scale=0.45]{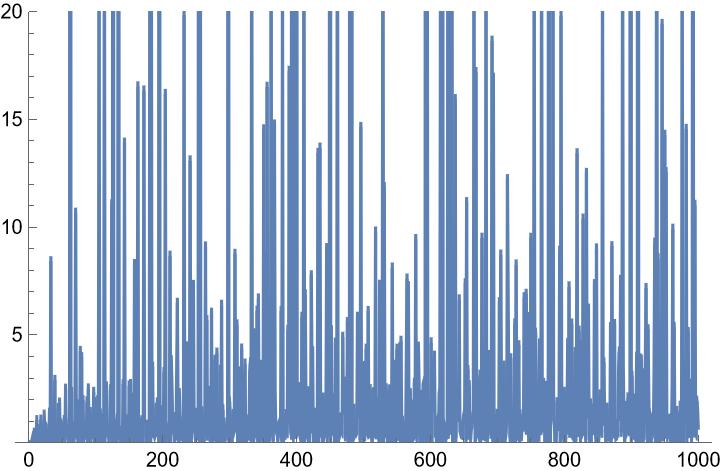} 	
		\caption{\small{Viscosity \(\mu(g_n)\) for the first \(n=1,\ldots,1000\) Gram points.}}
\label{fig:f51}
\end{figure}

As observed in Fig. \ref{fig:f51}, the values of the viscosity \(\mu(g_n)\) for general Gram points appear to be distributed without a discernible pattern or lower boundary. In comparison, Fig. \ref{fig:f52} illustrates an intriguing behaviour for the first $n=1,\ldots,1000$ bad Gram points, denoted as \(g_n^{bad}\).
\begin{figure}[ht!]
	\centering
		\includegraphics[scale=0.45]{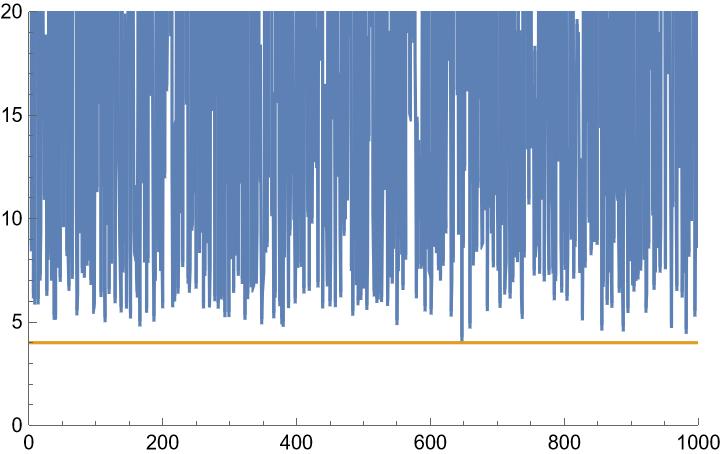} 	
		\caption{\small{Viscosity \(\mu(g^{bad}_n)\) for the first bad \(n=1,\ldots,1000\) Gram points, apparently bounded from below by a constant \(C > 4\).}}
\label{fig:f52}
\end{figure}

Remarkably, the data suggests that the viscosity \(\mu(g_n^{bad})\) for bad Gram points appears to be bounded from below by a constant \(C\), which is marginally greater than 4. This stands in contrast to the general case and represents what seems to be an unforeseen characteristic typical of bad Gram points. However, having observed this previously undiscovered bounded behaviour of \(\mu(g_n^{\text{bad}})\), we encounter another surprising layer of complexity as the computational range is extended further. In fact, as we proceed to higher values, a sparse subset of bad Gram points is discovered to sporadically defy this lower bound, yielding bad gram points with viscosity dramatically below $C$. We will refer to such unusual instances of bad Gram points as \emph{corrupt} Gram points. For instance, the 9807962-th Gram point is corrupt and its viscosity is 
\[ 
\mu(g_{9807962})=0.0750883.
\] 
Figure \ref{fig:f523} shows the viscosity \(\mu(g^{bad}_n)\) of all the bad Gram points arising between the $2.4 \cdot 10^7$-th and $2.43 \cdot 10^7$-th Gram points, with the corrupt points marked red: 
\begin{figure}[ht!]
	\centering
		\includegraphics[scale=0.3]{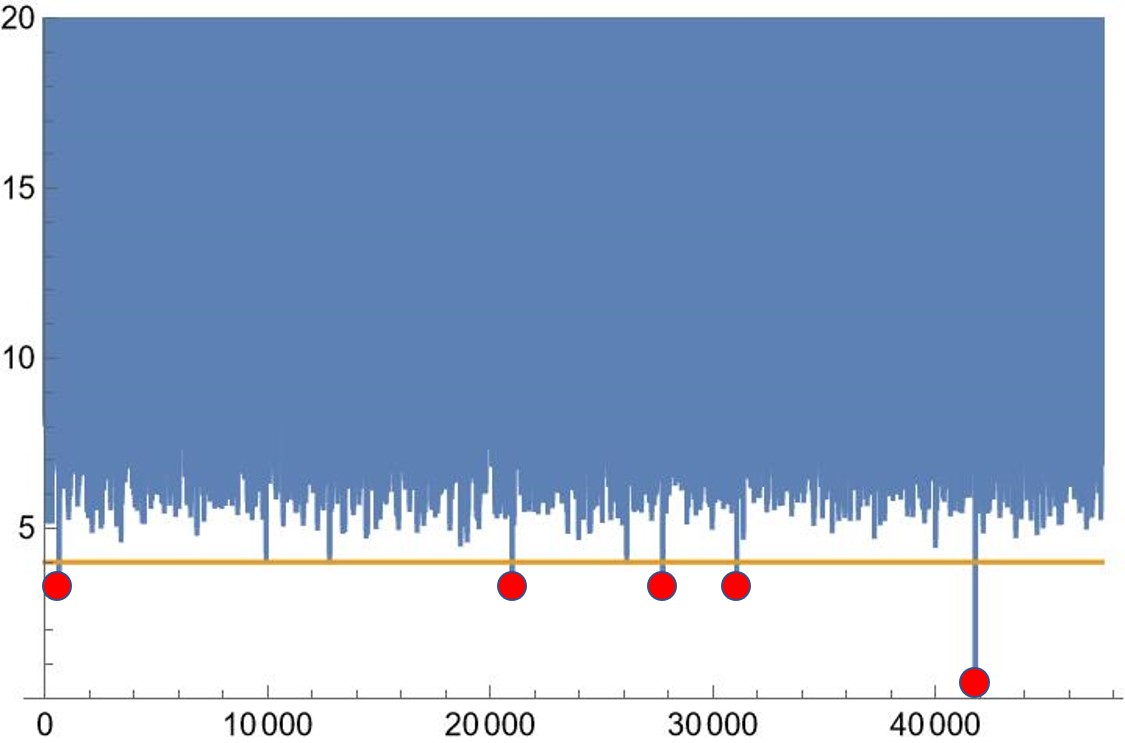} 	
		\caption{\small{Viscosity \(\mu(g^{bad}_n)\) of the bad Gram points between the $2.4 \cdot 10^7$-th and $2.43 \cdot 10^7$-th Gram points with corrupt points marked red.}}
\label{fig:f523}
\end{figure}

The emergence of such corrupt Gram points may initially seem contradictory to the insights we've developed thus far. However, it turns that these corrupt Gram points are not mere anomalies but exhibit a distinctive characteristic by themselves: they are actually observed to occur only under very specific conditions. Let us recall the following classic definition due to \cite{Ro} 
\begin{dfn}[Gram block]
A consecutive collection $\left \{ g_n,g_{n+1},...,g_{n+N} \right \}$ of Gram points is called a \emph{Gram block} if $g_n$ and $g_{n+N}$ are good Gram points while $g_{n+j}$ are bad Gram points for $j=1,...,N-1$. We refer to a bad Gram point as \emph{isolated} if it is the middle point of a block with $N=2$.  
\end{dfn} 	
For instance, the corrupt Gram point from the example above is part of the following Gram block of length $N=3$
\[ 
\left \{ 9807960,9807961,9807962,9807963 \right \}.
\]
Our experimental observations in this section are thus summarized in the following conjecture: 

\begin{conj}[G-B-G] \label{con:8.1}
Corrupt Gram points must be non-isolated. In particular, if a bad Gram point $g_n$ is in a triplet $\left \{ g_{n-1},g_n, g_{n+1} \right \}$ such that $g_{n-1}$ and $g_{n+1}$ are good Gram points, then $g_n$ cannot be corrupt.  
\end{conj}

We will see in the following sections that the bounded nature of \(\mu(g_n^{\text{bad}})\) for G-B-G points which is expressed in Conjecture \ref{con:8.1} and is observed empirically to extend substantially beyond the sample examples presented here, carries significant and profound implications for the study of Gram points as well as the distribution of zeros of the \(Z\)-function. Let us conclude this section with the following remark regarding notions of repulsion between consecutive zeros of the \(Z\)-function:
\begin{rem}[Repulsion and Montgomery's Conjecture] \label{rem:MPC}
The observed viscosity bound, \(\mu(g_n) > C\), implies that a large absolute value of \(Z(g_n)\) necessitates a large absolute value for \(Z'(g_n)\). Consequently, if there exists a force pushing the value of \(Z(t)\) at \(g_n\) towards the axis, there must also be a lateral force pushing it sideways. This suggests that, on an infinitesimal level, an extremal G-B-G point \(g_n(r)\) resists a change in sign. Given that a change in sign at \(g_n(r) \) would arise when two consecutive zeros collide, this viscosity bound unveils a previously unrecognized repulsion property between adjacent zeros of the \(Z\)-function. For \(\alpha \leq \beta\) define
\begin{equation}
A(T ; \alpha, \beta) := \left \{ (\rho,\rho') \mid 0 < \rho, \rho' < T \text{ and } \frac{2 \pi \alpha}{\ln(T)} \leq \rho - \rho' \leq \frac{2 \pi \beta}{\ln(T)} \right \}.
\end{equation} Montgomery's famous pair correlation conjecture (under the assumption of RH) states that
\begin{equation}
N(T ; \alpha, \beta) := \sum_A 1 \sim \left( \int_{\alpha}^{\beta} \left( 1 - \frac{\sin(\pi u)}{\pi u} \right) \, du + \delta_0([\alpha, \beta]) \right) \frac{T}{2 \pi} \ln(T)
\end{equation}
as \(T \to \infty\), see \cite{M}. Due to the decay of the integral for small \(u\), this conjecture is often interpreted as expressing statistical repulsion between consecutive zeros. 

It is crucial to note that while both our approach and Montgomery's conjecture discuss a form of repulsion between zeros, the nature of this repulsion is fundamentally different. In Montgomery's conjecture, the repulsion is statistical and encapsulates a property of zero distributions in the large scale. In contrast, our approach, guided by the dynamic nature of the A-philosophy, reveals a repulsion property which occurs at the infinitesimal level between specific pairs of consecutive zeros. This explicit, dynamic behaviour might be a foundational mechanism whose existence can be viewed as implicitly anticipated by Montgomery's statistical conjecture.
\end{rem}

 \section{The Failure of the G-B-G Property for the DH-Function }\label{s:9}
 The Davenport-Heilbronn functions are a specialized class of Dirichlet functions. Unlike the Riemann zeta function, these functions are known to have zeros that deviate from the critical line, thereby violating their corresponding RH property \cite{DH,T}. This enigmatic contrast between the behaviours of the Riemann zeta function and the Davenport-Heilbronn functions is often presented as a compelling illustration of the elusive nature of RH. We define the function as follows:

\begin{dfn}[DH-function]
We define the Davenport-Heilbronn function, \(\mathcal{D}(s)\), as
\[
\mathcal{D}(s)= \frac{(1-i \kappa)}{2} L(s,\chi_{5,2}) + \frac{(1+i \kappa)}{2} L(s,\overline{\chi}_{5,2}),
\]
where \(\kappa=\frac{\sqrt{10-2\sqrt{5}}-2}{\sqrt{5}-1}\).
\end{dfn}
The function satisfies the functional equation \( \xi(s)=\xi(1-s) \), where \[ \xi(s) = \left( \frac{\pi}{5} \right)^{-\frac{s}{2}} \Gamma \left( \frac{1+s}{2} \right) \mathcal{D}(s). \] To carry out an analysis analogous to our earlier treatment of the \( Z \)-function, we introduce the core function of \( \mathcal{D}(s) \):
\[
\mathcal{D}_0(s):=\frac{1}{2} \left[ \left( \frac{\pi}{5} \right)^{-\frac{s}{2}} \Gamma \left( \frac{1+s}{2} \right) + \left( \frac{\pi}{5} \right)^{\frac{s-1}{2}} \Gamma \left( \frac{2-s}{2} \right) \right].
\]
Similar to the \( Z \)-function, we can define the Davenport-Heilbronn \( Z \)-function, \( Z^{DH}(t) \), the \( A \)-space \( Z_N^{DH}(t; \overline{a}) \), and the Gram discriminant \( \Delta_n^{DH}(r) \). The zeros and extremal points of the core function \( Z_0^{DH}(t) \) are given by the following relations:
\begin{equation}
\begin{array}{ccc}
t^{DH}_n:=\frac{2 \pi (n - \frac{5}{8} ) }{W_0 (5 e^{-1}(n-\frac{5}{8} ))} & ; &
g^{DH}_n:=\frac{2 \pi (n - \frac{1}{8} ) }{W_0 (5 e^{-1}(n-\frac{1}{8} ))}.
\end{array}
\end{equation}

Our attention is particularly drawn to the first pair of zeros deviating from the critical line, which occur near the Gram point \( g^{DH}_{44} \). This deviation is depicted in Figure \ref{fig:f5}, where we plot the Gram discriminant \( \Delta^{DH}_{44}(r) \) and its first order approximation \(Z_N^{DH} (g_n^{DH}; r) \) over the interval \( 0 \leq r \leq 1 \).
\begin{figure}[ht!]
	\centering
		\includegraphics[scale=0.4]{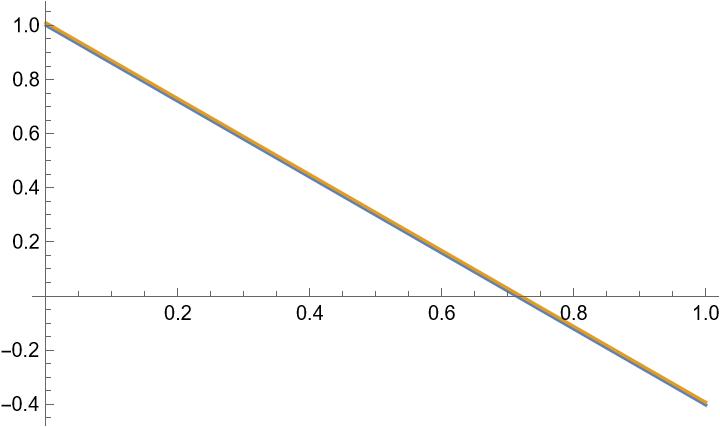} 	
		\caption{\small{\( \Delta^{DH}_{44} (r) \) (blue) and \(Z_N^{DH} (g_n^{DH}; r) \) (brown) for \( 0 \leq r \leq 1 \).}}
\label{fig:f5}
\end{figure}

The behaviour exhibited by $\Delta^{DH}_{44} (r)$ is seen to be a sort of mixture of the behaviour of $\Delta_{91}(r)$ and $\Delta_{126}(r)$ for the $Z$-function, presented in Example \ref{ex:1}. In the sense that, on the one hand, the first-order approximation $Z_N^{DH}(g_n^{DH} ; r)$ is seen to be in superb alignment with $\Delta^{DH}_{44}(r)$ (like for the good Gram point $g_{91}$) while, on the other hand, it is seen to nevertheless violate the Gram law (as for the bad Gram point $g_{126}$). 

In particular, Fig. \ref{fig:f5} shows that, in strict contrast to the case of the $Z$-function, for the Davenport-Heilbron function $Z^{DH}(t)$, the violation of Gram's law is not the result of non-linearity of the discriminant $\Delta_n^{DH}(r)$, but rather represents a genuine violation of the corrected Gram's law, as well. Note that this violation of the corrected Gram's law, in contrast to the case of the bad Gram point $g_{126}$, is the result of an unavoidable collision between the zeros $t^{DH}_{44}(r)$ and $t^{DH}_{45}(r)$ in the transition from $r=0$ and $r=1$, which will occur for any curve in the parameter space connecting the core $Z^{DH}_N(t;0)$ to $Z_N^{DH}(t;1)$. We thus have: 
	\begin{cor}
The Davenport-Heilbronn function $D(s)$ violates the corrected Gram's law. 

\end{cor}   
Figure \ref{fig:f6.1} shows the graphs of $Z^{DH}_N(t;r)$ in the range $t \in [g^{DH}_{44} -2,g^{DH}_{44}+2]$ for various values of $r \in [0,1]$:  	

\begin{figure}[ht!]
	\centering
		\includegraphics[scale=0.35]{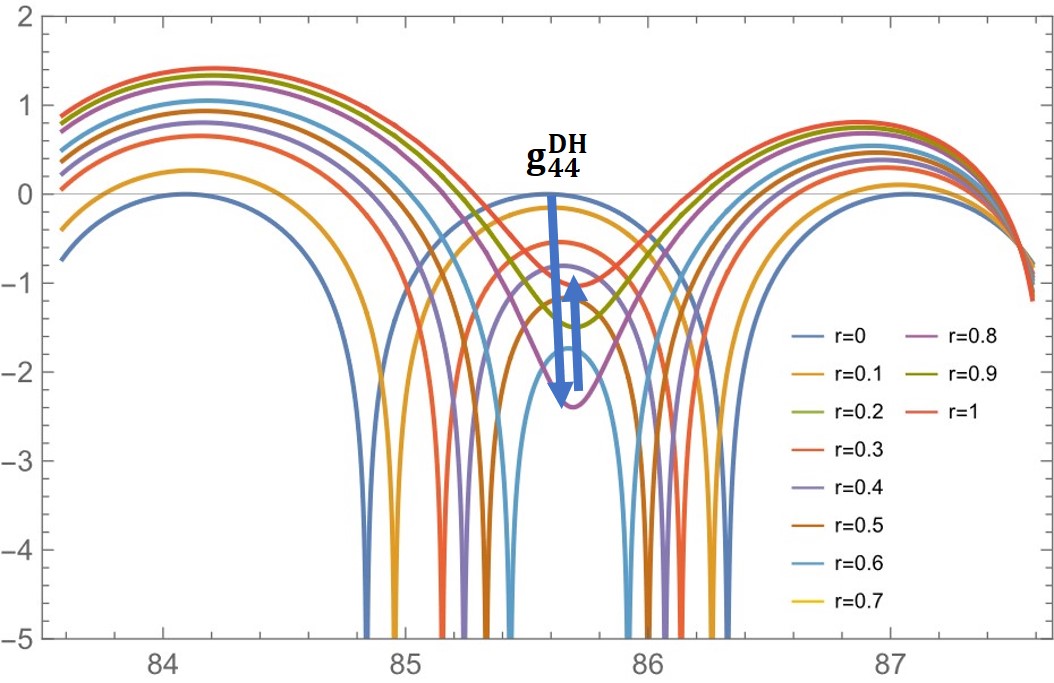} 	
		\caption{\small{Graphs of $Z^{DH}_N(t;r)$ in the range $t \in [g^{DH}_{44} -2,g^{DH}_{44}+2]$ for various values of $r \in [0,1]$.}}
\label{fig:f6.1}
	\end{figure} 
	
In particular, Fig. \ref{fig:f6.1} shows us that the extremal point $g_{44}^{DH}(r)$ itself presents a behaviour more similar to that of the good Gram point $g_{90}(r)$ rather than that of the bad Gram point $g_{126}(r)$, in the sense that it hardly exhibits a shift to the sides and remains rather almost fixed in its position as $r$ grows. In particular, the results of this section could be summarized in the following: 

\begin{cor}\label{cor:DH}
The Davenport-Heilbronn function $D(s)$ violates the G-B-G  repulsion property. 
\end{cor}   

It should be noted that a violation of the G-B-G property should be viewed as a more primal phenomena than a violation of RH, or the corrected Gram's law for that matter. It means that $Z^{DH}(t)$ lacks the regulatory property that $Z(t)$ is expected to have, of pushing to the sides bad Gram points, $g_n(r)$ as $r$ grows.

 \section{An In-Depth Investigation of The G-B-G Repulsion Property and Adjustments of $Z(g_{n \pm 1})$}\label{s:10}
 
 In Section \ref{s:8}, inspired by our discriminant analysis, we have experimentally discovered the G-B-G repulsion property according to which a bad Gram point $g_n$ with good consecutive neighbours $g_{n \pm 1}$ is expected to satisfy the viscosity bound $\mu(g_n)>4$. In this section, we undertake a formal exploration of the G-B-G property, aiming to uncover why the characteristics of \(g_n\) are anticipated to be interconnected with those of its consecutive neighbours. While a comprehensive proof of the G-B-G property is still beyond reach, subsequent sections will elucidate why we consider this property a significant stride towards a deeper understanding of the RH, and how it profoundly influences the properties of the zeros of the \(Z\)-function.

\subsection{The Cosine and Sine Adjustments of \( Z(g_{n \pm 1}) \)}
In order to explore why the values of \( Z'(g_n) \) and \( Z(g_n) \) are expected to be closely related to \( Z(g_{n-1}) \) and \( Z(g_{n+1}) \), we will turn to the classical Approximate Functional Equation (AFE), which is commonly used to compute the values of \( Z \) and \( Z' \) at Gram points. The AFE is represented by the following formulas:
\begin{equation}
\left\{
\begin{aligned}
    &Z(g_n) = 2 (-1)^n \sum_{k=1}^{N(n)} \frac{\cos(\ln(k) g_n)}{\sqrt{k}} + O\left( g_n^{-\frac{1}{4}} \right), \label{eq:Zgn} \\
    &Z'(g_n) = (-1)^n \sum_{k=1}^{N(n)} \ln\left(\frac{g_n}{2\pi k^2}\right) \frac{\sin(\ln(k) g_n)}{\sqrt{k}} + O\left( g_n^{-\frac{1}{4}} \right),
\end{aligned}
\right.
\end{equation}
where \( N(n) := \left [ \sqrt{\frac{g_n}{2\pi}} \right ] \) for any \( n \in \mathbb{Z} \). For a detailed derivation, see (5.2) and (6.3) in \cite{I}.

\begin{rem} In Remark \ref{rem:2.1}, we noted that we often compute $Z(t)$ using the approximation \eqref{eq:Z-function} with a sum of 
$\left [\frac{g_n}{2} \right]$ terms for the $A$-philosophy. However, in our current study focusing specifically on the values at 
$g_n$, the classical Approximate Functional Equation (AFE) with $N(n)$ terms is more suitable and will be employed.
\end{rem}

Let us introduce the following adjustments of $Z(g_{n \pm 1})$ to account for the influence of the adjacent Gram points \( g_{n-1} \) and \( g_{n+1} \) on $Z$ and $Z'$: 
\begin{dfn}[Cosine and Sine Adjustments of \(Z(g_{n \pm 1})\)] \label{def:10}
For any integer \( n \in \mathbb{Z}\), define:
\begin{enumerate}
    \item \( Z^{\pm}_c(g_n) = 2 (-1)^n \sum_{k=1}^{N(n)} \frac{\cos(\ln(k) g_{n \pm 1})}{\sqrt{k} \cos(\phi^n_k)} \),
    to which we refer as the cosine-adjustment of \( Z(g_{n \pm 1}) \).
    
    \item \( Z^{\pm}_s(g_n) = (-1)^n \sum_{k=1}^{N(n)} \ln\left(\frac{g_n}{2\pi k^2}\right) \frac{\cos(\ln(k) g_{n \pm 1})}{\sqrt{k} \sin(\phi_k^n)} \), to which we refer as the sine-adjustment of \( Z(g_{n \pm 1}) \).
\end{enumerate}
where for \(k = 1, \ldots, N(n)\) the adjustment-phase is given by 
\[
\phi_k^n := \ln(k)(g_n - g_{n-1}) = \ln(k) \frac{2\pi}{\ln \left( \frac{g_n}{2\pi} \right)}.
\] 
\end{dfn}

\begin{rem}[$\alpha$-adjustments of $Z(g_n)$] In general, for any sequence $\alpha = \alpha(n,k)$ let us define the $\alpha$-adjustments of $Z(g_n)$ with respect to $g_{n \pm 1}$ to be  
\[ Z^{\pm}(g_n ; \alpha ):= (-1)^n\sum_{k=1}^{N(n)} \alpha(n,k) \frac{\cos(\ln(k) g_{n \pm 1})}{\sqrt{k}}. \] 
The cosine and sine adjustments of \ref{def:10} correspond to the following two sequences
\[
\begin{array}{ccc} 
\alpha_c(n,k) = \frac{2}{\cos( \phi^n_k) } & ; & \alpha_s(n,k) = \ln \left ( \frac{g_n}{2\pi k^2} \right ) \frac{1}{\sin ( \phi_k^n) }. \end{array} \]

\end{rem}

The following result establishes the relationship between the values \(Z(g_n)\) and \(Z'(g_n)\) and the cosine and sine adjustments of $Z(g_{n \pm 1})$: 
\begin{prop} \label{prop:comb}
For any $n \in \mathbb{Z}$ the following holds: 
\begin{enumerate} 
    \item \( Z(g_n) = \frac{1}{2} \left ( Z_c^-(g_n)+Z_c^+(g_n) \right ) +O \left (g_n^{-\frac{1}{4}} \right )\).
    \item \( Z'(g_n) =  \frac{1}{2} \left ( Z_s^-(g_n)-Z_s^+(g_n) \right ) +O \left (g_n^{-\frac{1}{4}} \right )\).
\end{enumerate}
\end{prop} 
\begin{proof} The result follows from the identities 
\[ 
\begin{array}{ccc}
\cos ( \ln(k)g_n)  = \frac{\cos(\ln(k)g_{n-1})+\cos (\ln(k)g_{n+1})}{2 \cos(\phi_k^n)} & ; & 
\sin ( \ln(k)g_n)  = \frac{\cos(\ln(k)g_{n-1})-\cos (\ln(k)g_{n+1})}{2 \sin(\phi_k^n)} \end{array}
\] 
\end{proof} 
We are thus interested in the relation between $Z(g_{n \pm 1})$ and the adjustments $Z_c^{\pm}(g_n)$ and $Z_s^{\pm}(g_n)$. Let us investigate the properties of the phase $\phi^n_k$ and the $\alpha$-functions. Figure \ref{fig:f9} shows the graphs of $\alpha_c(n,k)$ and $\alpha_s(n,k)$ for $n=9807962$ and $k=1,...,N(n)$:

\begin{figure}[ht!]
	\centering
		\includegraphics[scale=0.45]{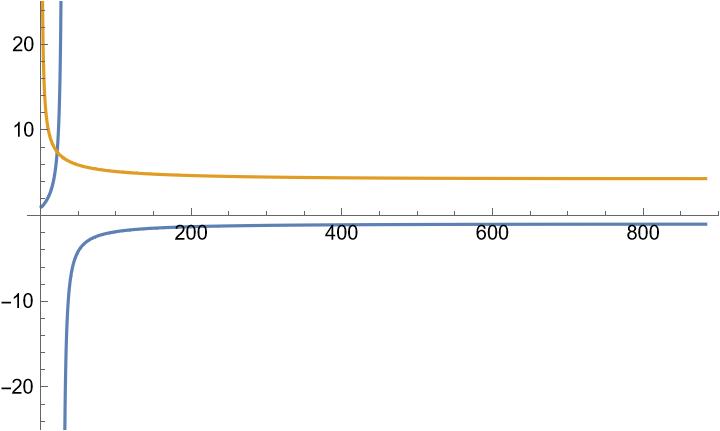} 	
		\caption{\small{$\alpha_c(n,k)$ (blue) and $\alpha_s(n,k)$ (brown) for $n=9807962$ and $k=1,...,N(n)$.}}
\label{fig:f9}
	\end{figure}
Figure \ref{fig:f10} shows the adjustment-phase $\phi_k^n$ for $n=9807962$ and $k=1,...,N(n)$:
\bigskip
\begin{figure}[ht!]
	\centering
		\includegraphics[scale=0.4]{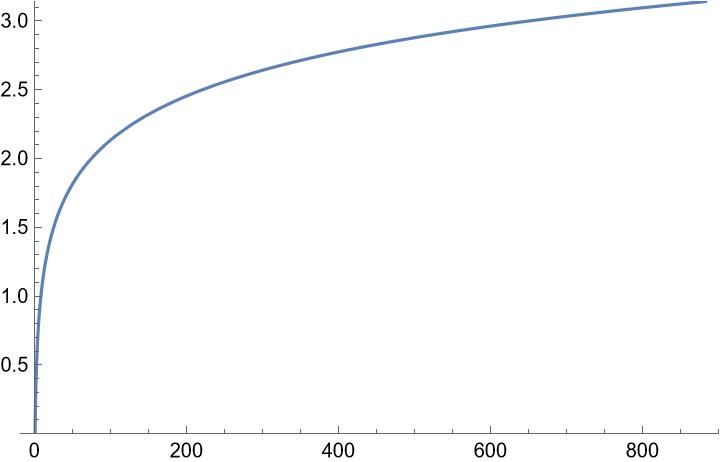} 	
		\caption{\small{The adjustment-phase $\phi_k^n$ for $n=9807962$ and $k=1,...,N(n)$.}}
\label{fig:f10}
	\end{figure}
	
The following describes the general properties of \( \alpha_c(n,k),\alpha_s(n,k) \) and \( \phi^n_k \) as continuous function of the variable $k$:

\begin{lemma}
\label{lem:10}
For any integer \( n \in \mathbb{Z} \), the functions \( \alpha_c(n, k) \) and \( \alpha_s(n, k) \) exhibit the following behaviors:
\begin{enumerate}
    \item The function \( \alpha_c(n, k) \) starts at \( \alpha_c(n, 0) = 2(-1)^n \) and increases monotonically on the interval \( 0 \leq k \leq \sqrt{N(n)} \). Furthermore, it satisfies 
    \[
    \lim_{{k \rightarrow \sqrt{N(n)}}^{\pm}} \alpha_c(n, k) = \mp \infty
    \]
    before rising to \( \alpha_c(n, N(n)) = -1 \).
    
    \item The function \( \alpha_s(n, k) \) starts at \( \alpha_s(n, 0) = + \infty \) and decreases monotonically until it reaches \( \alpha_s(n, N(n)) = \frac{1}{\pi} \ln \left( \frac{g_n}{2\pi} \right) \).
   \item The adjustment-phase \( \phi^n_k \) begins at \( \phi^n_1 = 0 \), rises to \( \phi^n_{N(n)} = \pi \), and attains a value of \( \phi^n_{\sqrt{N(n)}} = \frac{\pi}{2} \).
\end{enumerate}
\end{lemma}
\begin{proof} 
\( (1),(2) \) follow from \( (3)  \), which in turn follows from the definition of the phase via direct computation, the rest is immediate. 
\end{proof} 
\setcounter{subsection}{1}

\subsection{Approximating Adjustments via the Localized Sub-Sums of $Z(g_n)$} By definition, the adjustments modify the terms of the sums \(Z(g_{n \pm 1})\) through the \(\alpha(n,k)\)-function. While the classical sums \(Z(g_{n \pm 1})\) and their adjusted variants behave differently on a global scale, we observe that they become approximately proportional when summed over a restricted localized set of indices. This observation leads us to define:

\begin{dfn}[Localized Sums]
Let \(1 \leq a < b \leq N(n)\). The localized sum of \(Z(g_n)\) within the interval \([a, b]\) is given by:
\[
Z(g_n ; a, b) := 2 (-1)^n \sum_{k=a}^{b} \frac{\cos(\ln(k) g_n)}{\sqrt{k}}.
\]
Similarly, we introduce the localized sums for \(Z'(g_n)\) and \(Z^{\pm}_c(g_n), Z^{\pm}_s(g_n)\).
\end{dfn}

To quantify the average effect of the \(\alpha\)-function within an interval, we define:
\[
\alpha^{avg}(n ; a, b) := \frac{1}{b-a} \int_a^b \alpha(n, k) \, dk.
\]
As shown in Lemma \ref{lem:10} and depicted in Fig.~\ref{fig:f9}, we observe that both \( \alpha_c(n, k) \) and \( \alpha_s(n, k) \) approach constant values for sufficiently large \( k \) within the interval \( [1, N(n)] \). Consequently, their corresponding averages over an interval \( [a, N(n)] \) also approach these constant values, as $a$ grows. Specifically, we have:
\[
\label{eq:propor} 
\begin{array}{ccc}
\alpha_c^{avg}(n ; a, N(n)) \approx -1 & ; &
\alpha_s^{avg}(n ; a, N(n)) \approx \frac{1}{\pi} \ln \left ( \frac{g_n}{2 \pi} \right ),
\end{array}
\]
for relatively large \(1<a\), and the quality of the approximation improves as \(a\) increases. For instance, consider Fig.~\ref{fig:f11}, which illustrates the following graphs for \( n=9807962 \) and \( k = 1, \ldots, N(n) \):
\begin{itemize}
    \item[(a)] \( Z^-_c(g_n; k, N(n)) \) (blue), and \( -Z(g_{n-1} ; k, N(n)) \) (brown).
    \item[(b)] \( Z^-_s(g_n; k, N(n)) \) (blue), and \( \frac{1}{\pi} \ln \left ( \frac{g_n}{2 \pi} \right ) Z(g_{n-1} ; k, N(n)) \) (brown).
\end{itemize}
\begin{figure}[ht!]
	\centering
		\includegraphics[scale=0.4]{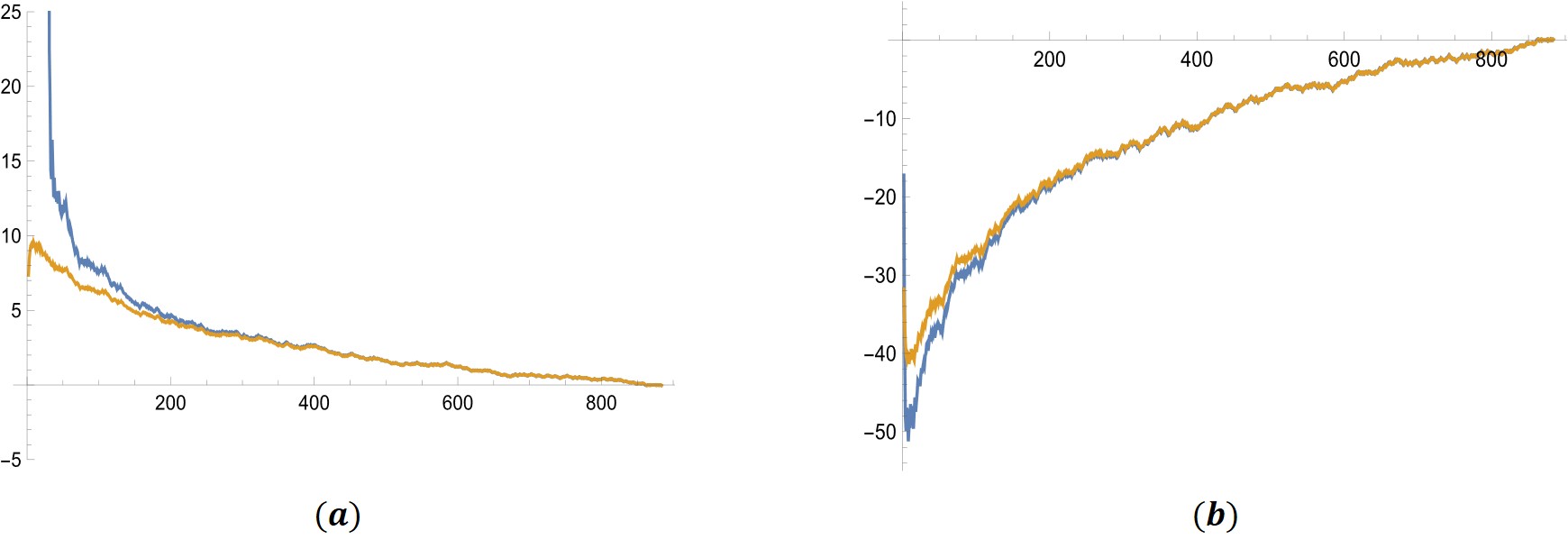} 	
		\caption{{\small Graphs of (a) \(Z^-_c(g_n; k, N(n))\) (blue) and \(-Z(g_{n-1} ; k, N(n))\) (brown).} {\small (b) \(Z^-_s(g_n; k, N(n))\) (blue) and \( \frac{1}{\pi} \ln \left ( \frac{g_n}{2 \pi} \right ) Z(g_{n-1} ; k, N(n))\) (brown).} For \(n=9807962\) and \(k=1,\ldots,N(n)\).}
\label{fig:f11}
	\end{figure}
As expected, Fig.~\ref{fig:f11} shows that the localized sums \( Z(g_{n \pm 1}) \) and their adjustments \( Z^{\pm}_c(g_n) \) and \( Z^{\pm}_s(g_n) \) are approximately proportional to each other, via the constants of ~\ref{eq:propor}. This proportionality holds within intervals of the form \([a, N(n)]\), with \( a \) is significantly larger than $1$ yet still much smaller than \( N(n) \). However, when the value of \( a \) is relatively small, the localized sums \( Z(g_{n \pm 1}) \) start to deviate from their adjustments \( Z^{\pm}_c(g_n) \) and \( Z^{\pm}_s(g_n) \), and the proportionality no longer holds. Nonetheless, for any given interval \([a, b]\), the following approximations can be considered:
\begin{equation}
\begin{aligned}
Z^{\pm}_c(g_n ; a, b) &\approx \alpha_c^{avg}(n ; a, b) \cdot Z(g_{n \pm 1} ; a, b), \\
Z^{\pm}_s(g_n ; a, b) &\approx \alpha_s^{avg}(n ; a, b) \cdot Z(g_{n \pm 1} ; a, b),
\end{aligned}
\end{equation}
The accuracy of these approximations increases as the size of the interval decreases. In practice, we find that these approximations remain relatively accurate even for intervals \( [a, b] \) close to the origin and larger than initially expected. We suspect this resilience is due to the fluctuations of the cosine terms within the sums. Consequently, the adjustments \( Z^{\pm}_c(g_n) \) and \( Z^{\pm}_s(g_n) \) can be expressed as composites of localized sums \( Z(g_{n \pm 1}) \), each scaled by its corresponding constant.

In particular, in order to approximate \(Z^{\pm}_c(g_n)\) and \(Z^{\pm}_s(g_n)\) in terms of localized sums of $Z(g_{n \pm 1})$, one can partition \([1, N(n)]\) into \(1 = a_1 < a_2 < \ldots < a_r = N(n)\) and consider:
\begin{equation}
\label{eq:adjust-approx} 
\begin{aligned}
Z^{\pm}_c(g_n) &\approx \sum_{k=1}^{r-1} \alpha_c^{\text{avg}}(n ; a_k, a_{k+1}) \cdot Z(g_{n \pm 1} ; a_k, a_{k+1}), \\
Z^{\pm}_s(g_n) &\approx \sum_{k=1}^{r-1} \alpha_s^{\text{avg}}(n ; a_k , a_{k+1}) \cdot Z(g_{n \pm 1} ; a_k, a_{k+1} ),
\end{aligned}
\end{equation}
which would give good approximations assuming the partition is chosen to be refined enough. Figure~\ref{fig:f12} provides a schematic representation of the relationships among \(Z(g_n)\), \(Z'(g_n)\), \(Z(g_{n \pm 1})\), and their adjustment terms \(Z^{\pm}_c(g_n)\) and \(Z^{\pm}_s(g_n)\) so-far discussed in this section.
\begin{figure}[h!]
\centering
\begin{tikzpicture}[>=Stealth]
    \tikzstyle{mycircle} = [draw=gray, circle, fill=gray!20, minimum size=1.5cm, inner sep=0pt]
    \tikzstyle{mysquare} = [draw=blue, rectangle, fill=blue!20, minimum size=1cm, inner sep=0pt]

    \node[mycircle] (W) at (0,0) {\(Z(g_{n-1})\)};
    \node[mycircle] (N) at (2.5,2.5) {\(Z(g_n)\)};
    \node[mycircle] (E) at (5,0) {\(Z(g_{n+1})\)};
    \node[mycircle] (S) at (2.5,-2.5) {\(Z'(g_n)\)};

    \node[mysquare] (NW) at (0,2.5) {\(Z^-_c(g_n)\)};
    \node[mysquare] (NE) at (5,2.5) {\(Z^+_c(g_n)\)};
    \node[mysquare] (SW) at (0,-2.5) {\(Z^-_s(g_n)\)};
    \node[mysquare] (SE) at (5,-2.5) {\(Z^+_s(g_n)\)};

    \draw[->] (W) -- (N);
    \draw[->] (W) -- (S);
    \draw[->] (E) -- (N);
    \draw[->] (E) -- (S);

    \draw[->] (NW) -- node[above] {\(+\)} (N);
    \draw[->] (NE) -- node[above] {\(+\)} (N);
    \draw[->] (SW) -- node[above] {\(+\)} (S);
    \draw[->] (SE) -- node[above] {\(-\)} (S);

    \draw[->] (W) -- node[left] {\(\alpha_c\)} (NW);
    \draw[->] (W) -- node[left] {\(\alpha_s\)} (SW);
    \draw[->] (E) -- node[right] {\(\alpha_c\)} (NE);
    \draw[->] (E) -- node[right] {\(\alpha_s\)} (SE);
\end{tikzpicture}
\caption{\small{Schematic representation of the relations between $Z(g_n),Z'(g_n), Z(g_{n \pm 1})$ and the adjustments $Z^{\pm}_c(g_n)$ and 
$Z^{\pm}_s(g_n)$.}} 
\label{fig:f12}
\end{figure}
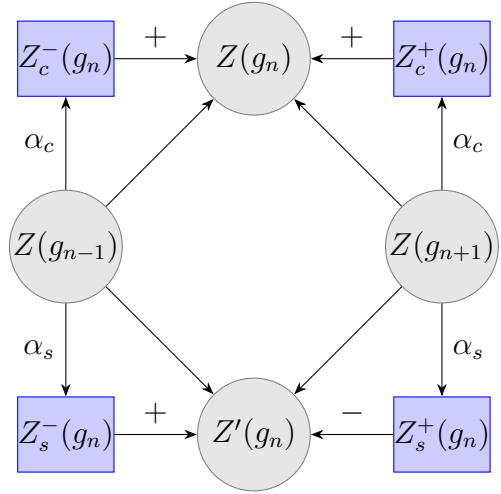

The perpendicular arrows represent the relationships established in Proposition~\ref{prop:comb}, which express \( Z(g_n) \) and \( Z'(g_n) \) as combinations of the adjustments \( Z^{\pm}_c(g_n) \) and \( Z^{\pm}_s(g_n) \). On the other hand, the vertical arrows depict the approximations given by ~\eqref{eq:adjust-approx}, enabling us to approximate the adjustments \( Z^{\pm}_c(g_n) \) and \( Z^{\pm}_s(g_n) \) through weighted partitions of \( Z(g_{n \pm 1}) \). Together, these results shed light on the expected intimate relationship between $Z(g_n)$, $Z'(g_n)$ and $Z(g_{n \pm 1})$. However, we have still not discussed the central question: why the consecutive Gram neighbours $g_{n \pm 1}$ being good or bad should have an effect on the proportion between $Z(g_n)$ and $Z'(g_n)$? 

\subsection{Further Numerical Study of the Repulsion Property and its Violations} In this sub-section, we seek to numerically investigate why the \emph{good} or \emph{bad} nature of the Gram neighbours \( g_{n \pm 1} \) is expected to have an effect on the proportion between \( Z(g_n) \) and \( Z'(g_n) \). Figure \ref{fig:f14} presents the partial sums \( Z(g_n ; 1,k) \) and \( Z'(g_n ; 1, k) \) as \( k \) ranges from 1 to \( N(n) \) for (a) the G-B-G bad Gram point at \( n=730119 \) and (b) the non-G-B-G bad Gram point at \( n=9807962 \):

\begin{figure}[ht!]
	\centering
		\includegraphics[scale=0.4]{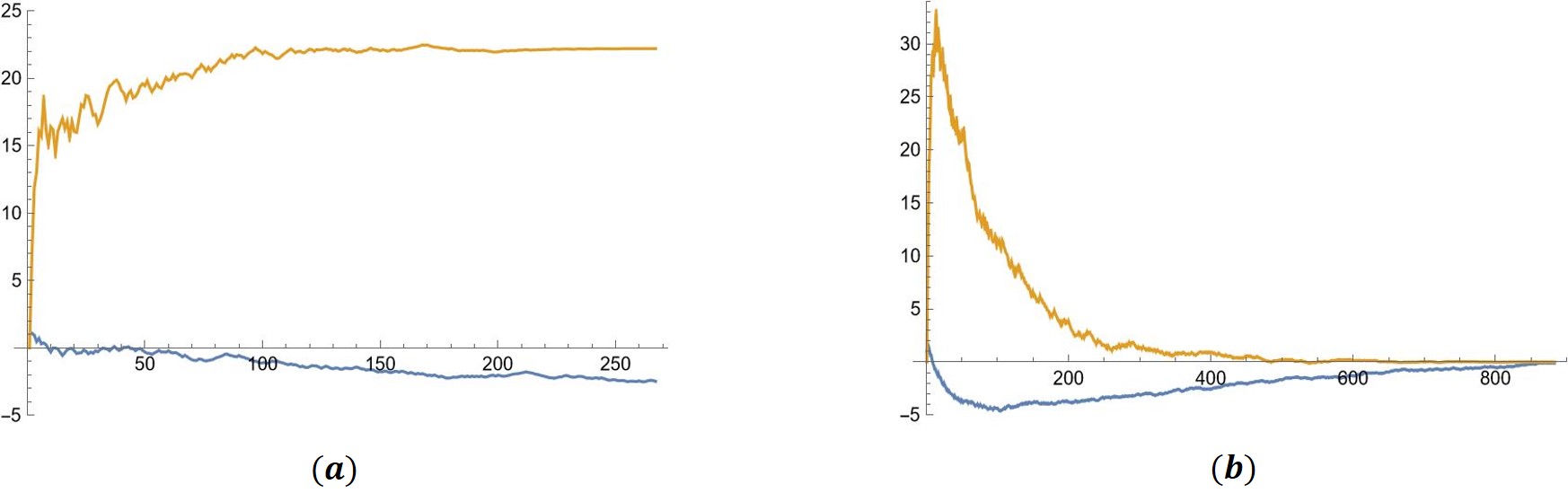} 	
		\caption{\small{The partial sums \( Z(g_n ; 1,k) \) and \( Z'(g_n ; 1, k) \) as \( k \) ranges from 1 to \( N(n) \) for (a) the G-B-G bad Gram point at \( n=730119 \) and (b) the non-G-B-G bad Gram point at \( n=9807962 \).}}
\label{fig:f14}
	\end{figure} Upon examining the partial sums of \( Z'(g_n) \), we identify three distinct stages across the index range \( k \) for both cases:

\begin{enumerate}
    \item[(a)] \emph{Initial surge:} Within the first \( \left\lceil  \sqrt[4]{\frac{g_n}{2 \pi}} \right\rceil \) terms, both cases exhibit a significant upward spike. This is attributed to the consistent sign of \( \cos(\ln(k)g_n) \) in this segment. Coupled with the relatively large values of \( \alpha_s(n,k) \) described by Lemma \ref{lem:10}, this creates a pronounced surge.

    \item[(b)] \emph{Middle fluctuation:} During this phase, the G-B-G point \( n=730119 \) displays a moderate upward shift following the initial impact. On the other hand, the non-G-B-G point \( n=9807962 \) experiences a substantial downward pull, nearly offsetting the initial surge.

    \item[(c)] \emph{Final stability:} As we progress to the latter part of the range, both sums stabilize and show minimal variation. As \( k \) nears \( N(n) \), the partial sums converge to \( Z'(g_n) \).
\end{enumerate}

Let us now explain a few of the observed properties, by utilizing our adjustments $Z^{\pm}_s(g_n)$. We have: 
\begin{lem} \label{lem:term1} The $k$-th term of the sum of $Z'(g_n)$ is given by 
$$ (-1)^n \frac{\alpha_s(n,k)}{\sqrt{k}} \left [ \cos ( \ln(k) g_{n-1} ) -\cos ( \ln(k) g_{n-1}+2\phi^n_k ) \right ],$$ where $\phi_k^n$ is the adjustment-phase. 
\end{lem} 

\begin{proof} Recall that by Proposition \ref{prop:comb} we have 
\[ Z'(g_n) =  \frac{1}{2} \left ( Z_s^-(g_n)-Z_s^+(g_n) \right ) \]
Hence, each individual term of $Z'(g_n)$ can be expressed as 
\begin{equation}
\label{eq:term}
 (-1)^n \frac{\ln \left ( \frac{g_n}{ 2\pi k^2} \right )}{\sqrt{k} sin ( \phi^n_k) } \left [ \cos ( \ln(k) g_{n-1} ) -\cos ( \ln(k) g_{n+1} )  \right ], \end{equation}
 and by definition $g_{n+1} = g_{n-1} + 2 \phi^n_k$. 
 \end{proof} 
 \hspace{-0.6cm} From the above, we can thus derive the following insights regarding the observed behaviour:
 \begin{enumerate}
 \item \emph{Final stability} By Lemma \ref{lem:10} we know that $\phi^n_k \approx \pi$ for $k$ close to the end of the range $N(n)$. Hence by Lemma \ref{lem:term1} for such indices the terms in this range satisfy:
 \begin{align}
  (-1)^n \frac{\alpha_s(n,k)}{\sqrt{k}}  \left [ \cos ( \ln(k) g_{n-1} ) -\cos ( \ln(k) g_{n-1}+2\phi^n_k ) \right ] \approx \\ \approx 
   (-1)^n \frac{\alpha_s(n,k)}{\sqrt{k}}  \left [ \cos ( \ln(k) g_{n-1} ) -\cos ( \ln(k) g_{n-1}+2\pi ) \right ] = \\ 
  =  (-1)^n \frac{\alpha_s(n,k)}{\sqrt{k}}  \left [ \cos ( \ln(k) g_{n-1} ) -\cos ( \ln(k) g_{n-1} ) \right ] =0, \end{align} which elucidates why the contribution of the terms becomes negligible as \(k\) approaches \(N(n)\).

\item \emph{Middle Fluctuations:} Lemma \ref{lem:term1} shows that the sign of the $k$-th term is fundamentally determined by the difference between $\cos(\ln(k) g_{n-1})$ and $\cos(\ln(k) g_{n+1})$. Consider, for instance, our specific case of $n=9807962$, which is a non-G-B-G point. In particular, in this case the left neighbour $g_{n-1}$ is also a bad Gram point. By definition, a Gram point is bad if 
\( \sum_{k=1}^{N(n)} \frac{\cos (\ln(k) g_n) }{\sqrt{k} } <0 \). This entails that the values of $\frac{\cos(\ln(k) t)}{\sqrt{k}}$ for $t= g_n$ and $t= g_{n-1}$ possess a significant tendency for negativity. We'll delve deeper into this notion shortly, in sub-Section \ref{sub:sim}.  In our scenario, we observe that many of the negative values for $g_{n-1}$ turn to cluster in the middle region, pulling the partial sums back towards the axis. This becomes especially evident within the range:
\begin{equation}
\label{eq:range}
 \left [ \frac{1}{2}  \sqrt[4]{\frac{g_n}{2 \pi}} \right ] \leq k \leq   \left [ 2 \sqrt[4]{\frac{g_n}{2 \pi}} \right ], 
\end{equation} This specific range is highlighted in Figure \ref{fig:f15}:
\begin{figure}[ht!]
	\centering
		\includegraphics[scale=0.45]{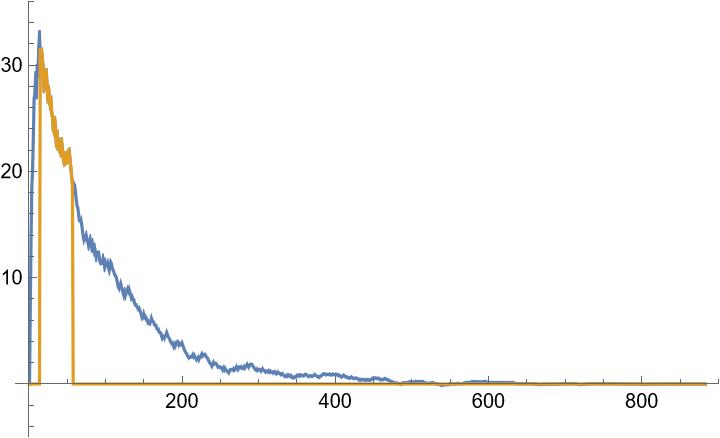} 	
		\caption{\small{The partial sums \( Z'(g_n ; 1, k) \) as \( k \) varies from 1 to \( N(n) \) for \( n=9807962 \). The region defined by \eqref{eq:range} is accentuated in orange.}}
\label{fig:f15}
\end{figure}

In this range, the adjustment-phase is given approximately by $2 \phi^n_k \approx \pi$ according to Lemma \ref{lem:10}. Consequently, within this range, the terms of $Z'(g_n)$ are approximately given by:
\begin{align}
  (-1)^n \frac{\alpha_s(n,k)}{\sqrt{k}} \left [ \cos ( \ln(k) g_{n-1} ) - \cos ( \ln(k) g_{n-1}+2\phi^n_k ) \right ] &\approx \\
  (-1)^n  \frac{\alpha_s(n,k)}{\sqrt{k}}  \left [ \cos ( \ln(k) g_{n-1} ) - \cos ( \ln(k) g_{n-1}+\pi ) \right ] &= \\
  2(-1)^n  \frac{\alpha_s(n,k)}{\sqrt{k}}  \cos ( \ln(k) g_{n-1} ), 
\end{align}
according to Lemma \ref{lem:term1}. This expected behaviour is indeed illustrated in the following Fig. \ref{fig:f15.1}:
\begin{figure}[ht!]
	\centering
		\includegraphics[scale=0.45]{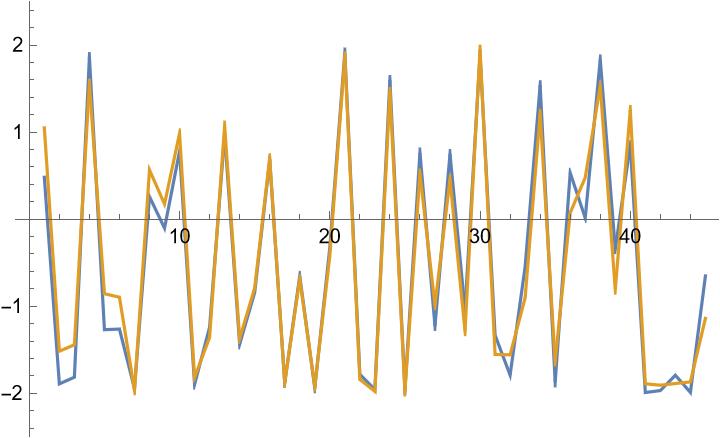} 	
		\caption{\small{$2 cos(ln(k) g_{n-1})$ (blue) and  $ cos(ln(k) g_{n-1})-cos(ln(k) g_{n+1})$ (orange) in the region defined by \eqref{eq:range}.}}
\label{fig:f15.1}
\end{figure}

The values of $2 cos(ln(k) g_{n-1})$ are seen to be predominantly negative which causes the observed decrease in the values of the partial sums of $Z'(g_n)$ in this region. 
\item  \emph{Initial range:} The main feature of the repulsion conjecture necessitates that in an initial surge exists then the moderate behaviour in the middle range must follow. That is, that there exists a strong correlation between the values of $cos(ln(k) g_n)$ in the range up to $\sqrt{N(n)}$ and those from $\sqrt{N(n)}$ to $N(n)$.  We have yet to pinpoint this correlation in the context of \( Z(t) \), but it is worth noting that for individual terms, such a correlation is inherently evident due to basic trigonometric identities. For instance, the double-angle formula for cosine 
\[
\cos(\ln(k^2) g_n) = \cos(2\ln(k) g_n) = 2\cos^2(\ln(k) g_n) - 1 
\]
illuminates that terms from \( [1, \sqrt{N(n)}] \) naturally correlate with those in the range \( [\sqrt{N(n)}, N(n)] \). However, this direct correlation via the double-angle identity maps terms from \( [1, \sqrt{N(n)}] \) to a sparse subset of \( [\sqrt{N(n)}, N(n)] \) through \( k \mapsto k^2 \). Unravelling why a profound correlation should prevail, impacting the collective values of terms across these two distinct ranges, is a complex task. A future thorough examination, likely incorporating statistical and probabilistic approaches, is essential for a more comprehensive understanding.

\end{enumerate}

Figure \ref{fig:f16} presents the partial sums \( Z(g_n ; 1,k) \) and \( Z'(g_n ; 1, k) \) as \( k \) ranges from 1 to \( N(n) \) for the good Gram point $n=195644$:

\begin{figure}[ht!]
	\centering
		\includegraphics[scale=0.45]{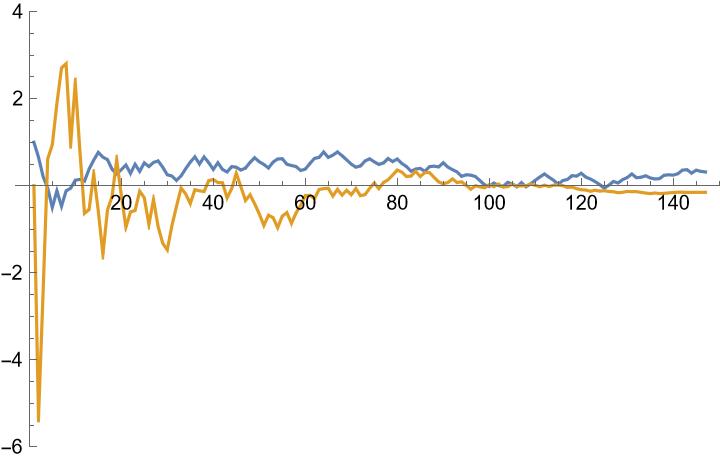} 	
		\caption{{\small the partial sums \( Z(g_n ; 1,k) \) and \( Z'(g_n ; 1, k) \) as \( k \) ranges from 1 to \( N(n) \) for the good Gram point $n=195644$}.}
\label{fig:f16}
	\end{figure}
	
	In Fig. \ref{fig:f16}, a distinctly different behaviour is exhibited when contrasted with the bad Gram points we previously considered. For the good Gram point $n=195644$, there is no initial surge in the partial sums of $Z'(g_n)$. Instead, the initial terms seem to counterbalance each other. Furthermore, within the middle range, there's only a modest shift in the partial sums. As a result, the value of $Z'(g_n)$ remains relatively small, especially when compared to $Z(g_n)$.

Finally, we should highlight that the initial surge observed in the partial sums $Z'(g_n)$ for the G-B-G Gram point $n=730119$ is common but not a universal occurrence. There are exceptions, as depicted in Fig. \ref{fig:f17}, which presents the partial sums \(Z(g_n; 1, k)\) and \(Z'(g_n; 1, k)\) for the G-B-G bad Gram point $n=300894$, where the surge in the initial range is notably absent. In this figure, as \(k\) varies from 1 to \(N(n)\), a different pattern of behaviour in the partial sums is clearly illustrated.
\begin{figure}[ht!]
	\centering
		\includegraphics[scale=0.45]{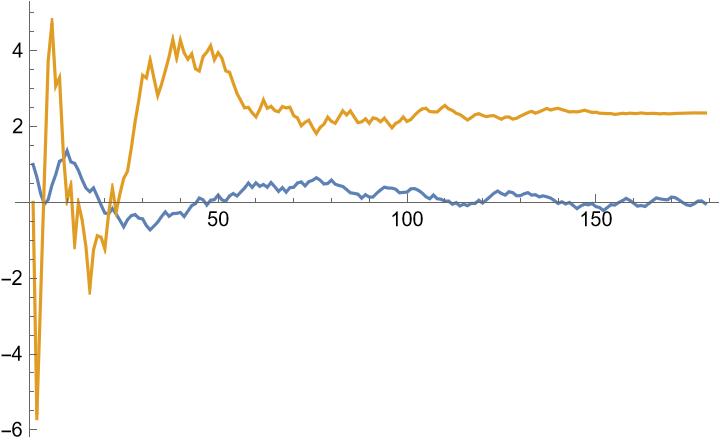} 	
		\caption{{\small the partial sums \( Z(g_n ; 1,k) \) and \( Z'(g_n ; 1, k) \) as \( k \) ranges from 1 to \( N(n) \) for the G-B-G Gram point $n=300894$}.}
\label{fig:f17}
	\end{figure}
	
As illustrated in Fig. \ref{fig:f17}, there is an absence of an initial surge and the terms appear to counterbalance each other, reminiscent of the behaviour observed for the good Gram point $n=195644$ shown earlier. However, a surge is seen eventually to emerge further along the range, as expected by the G-B-G repulsion conjecture.

To summarize, we have conducted an in-depth analysis of the fundamental relationships between the four components of the G-B-G conjecture, which are $Z(g_{n-1}), Z(g_n),Z(g_{n+1})$ and $Z'(g_n)$. We further conducted a numerical investigation of the repulsion property in various cases. We saw two cases of a G-B-G and non-G-B-G Gram point which were a-priori chosen to present an initial surge, and distinguished between the ability of a non-G-B-G point to counteract the surge, contrary to the moderate middle behaviour of the G-B-G point. However, we are not yet able to explain the other direction, which lies at the heart of the G-B-G conjecture and implies that if a strong pull towards the axis exists in the middle range than an even larger surge must initially occur. In particular, although the investigation conducted here sheds some light on the phenomena the main mystery still remains at this point: what causes the forced correlation between $Z(g_n)$ and $Z'(g_n)$ observed in the G-B-G case? In order to further investigate this question, it's essential delve deeper into the factors that determine whether a Gram point is good or bad.

\subsection{Monte-Carlo Simulation for Good and Bad Gram Points} \label{sub:sim} 
Let us further investigate the distribution of negative values within the vector \[ 
\mathbf{v}_n = \begin{bmatrix}
\frac{\cos(\ln(1) g_n)}{\sqrt{1}} \\
\frac{\cos(\ln(2) g_n)}{\sqrt{2}} \\
\vdots \\
\frac{\cos(\ln(N(n)) g_n)}{\sqrt{N(n)}}
\end{bmatrix} \in \mathbb{R}^{N(n)},
\]
for good and bad Gram points. As a study case, Fig. \ref{fig:f18} illustrates the vector $\mathbf{v}_n$ for the bad Gram point $n=730119$ and the good Gram point $n=730120$:
\begin{figure}[ht!]
	\centering
		\includegraphics[scale=0.4]{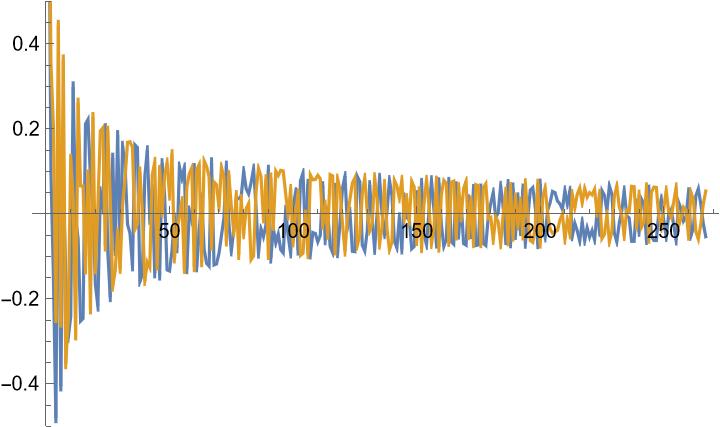} 	
		\caption{\small{The entries of the vector  $\mathbf{v}_n \in \mathbb{R}^{N(n)}$ for the bad Gram point \( n=730119 \) (blue) and the good Gram point \( n=730120 \) (orange).}}
\label{fig:f18}
	\end{figure}
	
From the illustrated vector \( \mathbf{v}_n \), distinguishing between a good and a bad Gram point is not immediately apparent. To better discern the distinction between good and bad Gram points, one might consider sorting the vector \(\mathbf{v}_n\) in ascending order, from its smallest to its largest entries. Let's denote this sorted vector by \(\mathbf{v}_n^{\mathrm{sorted}}\). Fig. \ref{fig:f19} illustrates the vector \(\mathbf{v}_n^{\mathrm{sorted}}\) for the bad Gram point $n=730119$ and the good Gram point $n=730120$:
\begin{figure}[ht!]
	\centering
		\includegraphics[scale=0.425]{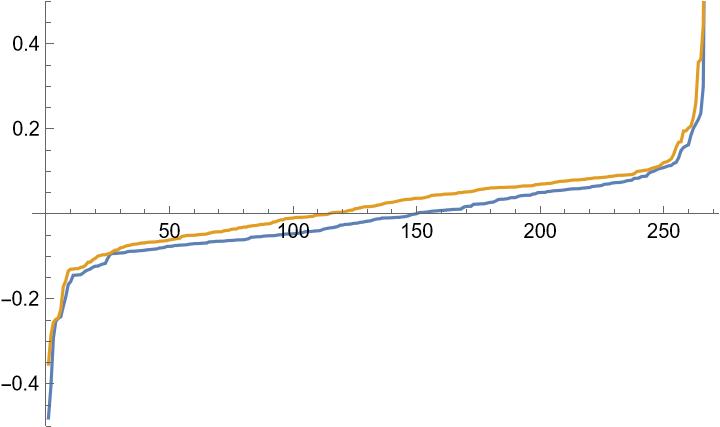} 	
		\caption{\small{The entries of the sorted vector \(\mathbf{v}_n^{\mathrm{sorted}}\) for the bad Gram point \( n=730119 \) (blue) and the good Gram point \( n=730120 \) (orange).}}
\label{fig:f19}
	\end{figure}
	
	In Fig. \ref{fig:f19}, a clear pattern differentiating between good and bad points remains elusive. Actually, both the vectors corresponding to good and bad points appear to exhibit a broadly similar structure. The pertinent question then arises: what characterizes this underlying structure? In order to investigate it we introduce the following: 
\begin{dfn}[Monte-Carlo Simulated Gram Vector]
For every $n \in \mathbb{Z}$, define the randomized $n$-th vector as
\[ 
\mathbf{v}^{random}_n = \begin{bmatrix}
\frac{\cos(\theta_1)}{\sqrt{1}} \\
\frac{\cos(\theta_2)}{\sqrt{2}} \\
\vdots \\
\frac{\cos(\theta_{N(n)})}{{\sqrt{N(n)}}}
\end{bmatrix} \in \mathbb{R}^{N(n)},
\]
where each $\theta_k$ is a random variable uniformly distributed over $[0,2\pi]$ for $k=1,\ldots,N(n)$. Let \( \mathbf{v}^{sorted-random}_n \) be the vector obtained by sorting the entries of \(\mathbf{v}^{random}_n \). The \emph{Monte-Carlo simulated $n$-th Gram vector} is defined as
the expectation \[
\mathbf{v}^{Monte-Carlo}_n := \mathbb{E}(\mathbf{v}^{sorted-random}_n).
\]
\end{dfn}

Essentially, the $n$-th Monte-Carlo vector $\mathbf{v}^{Monte-Carlo}_n$ is expected to offer the "baseline behaviour" of $\mathbf{v}^{sorted}_n$. Figure \ref{fig:f20} shows the entries of the vector \(\mathbf{v}_n^{\mathrm{sorted}}\) for the bad Gram point \( n=730119 \) (blue) and the good Gram point \( n=730120 \) (orange) together with \(\mathbf{v}_n^{\mathrm{Monte-Carlo}}\) (green):
\begin{figure}[ht!]
	\centering
		\includegraphics[scale=0.425]{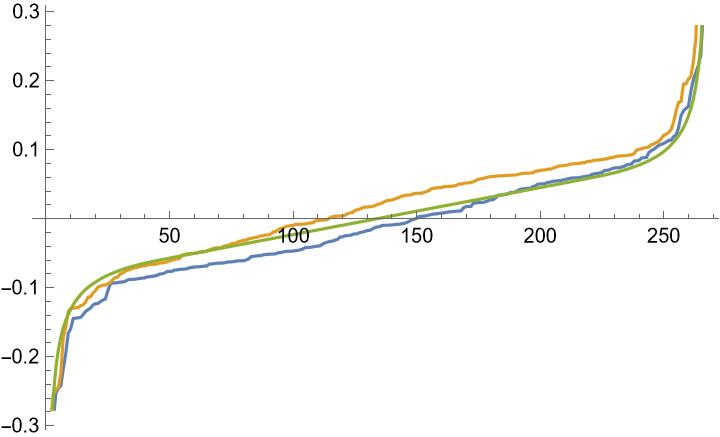} 	
		\caption{\small{The entries of \(\mathbf{v}_n^{\mathrm{sorted}}\) for the bad Gram point \( n=730119 \) (blue) and the good Gram point \( n=730120 \) (orange) together with \(\mathbf{v}_n^{\mathrm{Monte-Carlo}}\) (green)}.}
\label{fig:f20}
	\end{figure}

Since we view $\mathbf{v}^{Monte-Carlo}_n$ as representing the base standardized behaviour of $g_n$, we accordingly expect that the difference \[\mathbf{v}_n^{\mathrm{essential}}:=\mathbf{v}_n^{\mathrm{sorted}}-\mathbf{v}_n^{\mathrm{Monte-Carlo}}\]  would reflect the inherent properties of the Gram point. We have: 
\begin{prop} 
For every $n \in \mathbb{Z}$ the following holds: 
\begin{enumerate}
\item $\Sigma ( \mathbf{v}_n^{\mathrm{Monte-Carlo}}) =0$. 
\item $ \Sigma ( \mathbf{v}_n^{\mathrm{essential}})= (-1)^n Z(g_n)$.
\end{enumerate}
Where \( \Sigma(\mathbf{v} ) := \sum_{k=1}^{N(n)} \mathbf{v}_i, \) for a vector $\mathbf{v} \in \mathbb{R}^{N(n)}$. 
\begin{proof} 
We have:
\begin{enumerate}
\item Follows immediately from the anti-symmetry of the Monte-Carlo vector $\mathbf{v}_n^{\mathrm{Monte-Carlo}}$. 
\item From linearity
\begin{align*}
\Sigma ( \mathbf{v}_n^{\mathrm{essential}})=\Sigma ( \mathbf{v}_n^{\mathrm{sorted}}- \mathbf{v}_n^{\mathrm{Monte-Carlo}})=\Sigma ( \mathbf{v}_n^{\mathrm{sorted}})- \Sigma(\mathbf{v}_n^{\mathrm{Monte-Carlo}})=\\= \Sigma ( \mathbf{v}_n^{\mathrm{sorted}}) = \Sigma ( \mathbf{v}_n)=(-1)^n Z(g_n).
\end{align*}
 \end{enumerate}
\end{proof} 
\end{prop}

Fig. \ref{fig:f21} shows \(\mathbf{v}_n^{\mathrm{essential}}\) for \( n=730119 \) (blue) and \( n=730120 \) (orange): 
	\newpage
	
	\begin{figure}[ht!]
	\centering
		\includegraphics[scale=0.425]{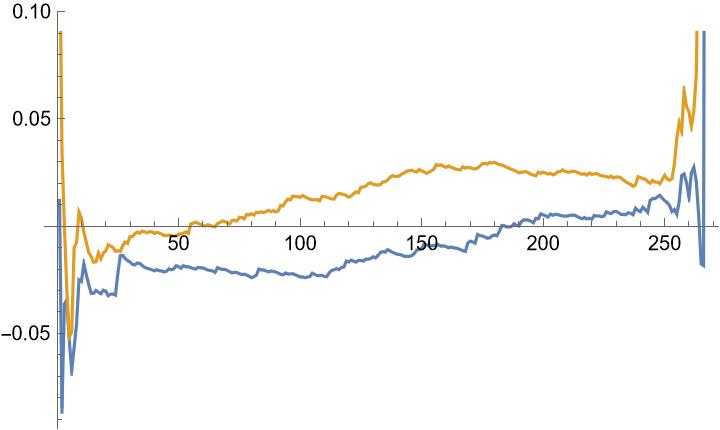} 	
		\caption{\small{\(\mathbf{v}_n^{\mathrm{essential}}\) for the bad Gram point \( n=730119 \) (blue) and the good Gram point \( n=730120 \) (orange).}}
\label{fig:f21}
	\end{figure}

Remarkably, \(\mathbf{v}_n^{\mathrm{essential}}\) indeed reveals the essential unique chaotic features of the Gram point $g_n$, after removing from it the structured baseline. In particular, one can discern from the graph the abundant negativity of the bad Gram point \( n=730119 \) as compared to the abundant positivity of the good Gram point  \( n=730120 \). In light of the above, we propose that a localized adaptation of the Monte-Carlo analysis presented herein, when applied to \(Z(g_n)\), \(Z'(g_n)\), and their respective adjustments, may offer a constructive direction for further investigation and possible future proof of the repulsion property.
\section{Edwards' Speculation on RH and its Shortcomings}
\label{s:11}

Up to this point, we have studied the discriminant of the linear curve $\Delta_n(r)$ from an infinitesimal point of view, focusing on its properties in a neighbourhood of $r=0$. However, our ultimate aim is to discuss the corrected Gram's law, which is principally concerned with the value of $r=1$. In the following, we will elucidate how our local analysis can be pushed forward to lead to meaningful insights on the value at $r=1$.

 In order to explain our approach for this transition from local to global, based on the $A$-space, let us return to the rough relation between the zeros of the core function $Z_0(t)=\cos(\theta(t))$ and the zeros of $Z(t)$, observed in Section \ref{s:2}. As mentioned, this relation has been noted by various authors over the years \cite{E,FL,J,SP1}. Particularly, Edwards, in Section 7.8 of his seminal work \cite{E}, presents some intriguing "speculations" on the possible origins of the Riemann hypothesis. He even proposes that Riemann might have noticed the numerical fact that the zeros of $\cos(\theta(t))$ serve as an initial, crude approximation of the zeros of $Z(t)$, especially for sufficiently small $t$ values. 
 
 In this section, we seek to examine Edwards' approach in the context of our methodology, identify its shortcomings, and offer a more sensitive generalization based on the $A$-philosophy. Edwards conjectures that the roots of the RH may be traced back to a rather heuristic concept, described by him as the idea that \emph{"one could go from a zero $t^0_n$ of $\cos(\theta(t))$ to a zero $t_n$ of $Z(t)$"}
via some iterative numerical process. However, Edwards remains quite ambiguous about the nature of this process and expresses substantial scepticism about the existence of a universally applicable procedure of this kind. He particularly points out the
extreme failure of Gram's law, for example the failure between $g_{6708}$ and $g_{6708}$ in Lehemer's graph as a key source of his doubts. 

In classical Newton's method, one typically starts with a crude initial guess $t^0$ for a zero of a function $F(t)$, then applies the iterative process
\[
t^{k+1} = t^{k} - \frac{F(t^{k})}{F'(t^{k})},
\]
aiming to locate a true zero, given by $t = \lim_{k \rightarrow \infty} (t^k)$, see for instance \cite{DB,SM}. Consequently, the following question naturally arises:
\begin{myquestion}[Edwards' speculation - Naive version] 
For $n \in \mathbb{Z}$ let $t^0_n$ be the $n$-th zero of the core function $Z_0(t)$. Can we use Newton's method starting from $t^0_n$ to consistently converge to $t_n$, the $n$-th real zero of $Z(t)$?
\end{myquestion}
Equivalently, we have: 
\begin{myquestion}[Corrected Gram's law - Naive version] For $n \in \mathbb{Z}$ let $g_n$ be the $n$-th Gram point (extremal point of $Z_0(t)$). Can we use Newton's method starting from $g_n$ to consistently converge to $\widetilde{g}_n$, the $n$-th real extremal point of $Z(t)$ and show that $(-1)^n Z(\widetilde{g}_n)>0$?
\end{myquestion}

Consider the following example: 

\begin{ex}[The Lehmer points $t_{6708}$ and $t_{6709}$] \label{ex:2} Let us apply Newton's method, starting from $t^0_n$, to the two Lehmer points referred to by Edwards. These points are approximately given by
\begin{equation}
\begin{array}{ccc}
t_{6708}\approx 7005.06 & ; & t_{6709}\approx 7005.10.
\end{array}
\end{equation}
The first Newton iterations $t^k_{6708}$ and $t^k_{6709}$ for these two zeros are presented in Table \ref{table:1}:
\begin{table}[htbp]
\centering
\begin{tabular}{||c c c ||} 
 \hline
 k & $t^k_{6708}$ & $t^k_{6709}$  \\ [0.5ex] 
 \hline\hline
 0 & 7004.95 & 7005.84  \\ 
 \hline
 1 & 7005.01 & 7005.23 \\
 \hline
 2 & 7005.04 & 7005.15  \\
 \hline
 3 & 7005.05 & 7005.12 \\
 \hline
 4 & 7005.06 & 7005.10 \\
  [1ex] 
 \hline
\end{tabular}
\caption{\small{Newton's Iterations for $t^{k}_{6708}$ and $t^{k}_{6709}$}}
\label{table:1}
\end{table}

    From the table, it is clear that despite the closeness of the Lehmer pair $t_{6708}$ and $t_{6709}$, the iterations of $t^k_{6708}$ and $t^k_{6709}$ do converge to $t_{6708}$ and $t_{6709}$, respectively. From our point of view, the following Fig. \ref{fig:f8.1} shows the graphs of $\Delta_{6708} (r)$ (blue) and $Z_N (g_{6708} ; r)$ (orange) with $0 \leq r \leq 1$:
    
\begin{figure}[ht!]
	\centering
		\includegraphics[scale=0.4]{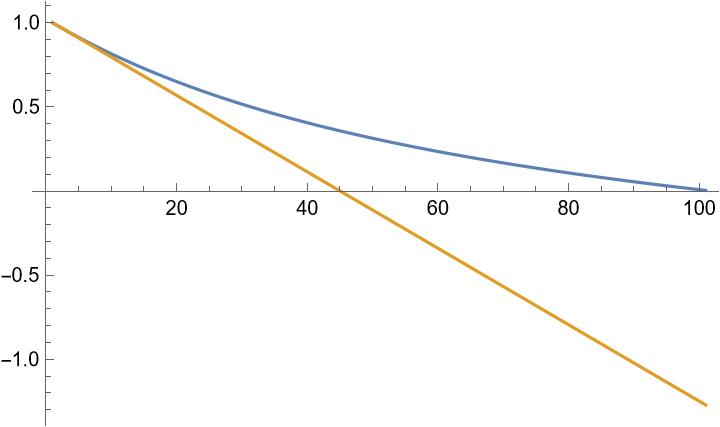} 	
		\caption{\small{Graph of the discriminant $\Delta_{6708} (r)$ (blue) and $Z_N (g_{6708} ; r)$ (orange) with $0 \leq r \leq 1$.}}
\label{fig:f8.1}
	\end{figure}
	
	Despite the extreme violation of Gram's law, as exemplified by the negative value of $Z_N (g_{6708} ; r)$, the corrected version of Gram's law still holds, as demonstrated by the graph of the discriminant $\Delta_{6708} (r)$. Let us note that $g_{6708}$ is an isolated bad Gram point with an actually relatively large viscosity $\mu(g_{6708})=6.41706$. 
 \end{ex}

However, although Newton's method appears successful in the above example, Edwards' concerns are actually far from being unjustified. Notably, even in algebraic cases, the convergence of Newton's method is well known to be closely tied to the properties of the discriminant. Hence, in our case, let us look for a violation of Newton's method where the discriminant is anomalous. The following is an example of an isolated bad Gram point with small viscosity, showing failure of Newton's method: 

\begin{ex}[Failure of Newton's method near $g_{730119}$] \label{ex:g} Consider $g_{730119}$ which is an isolated bad Gram point with relatively very small viscosity $\mu(g_{730119})=4.46023$. Figure \ref{fig:f8.2} shows the graphs of $\Delta_{730119} (r)$ (blue) and $Z_N (g_{730119} ; r)$ (orange) with $0 \leq r \leq 1$:
 \begin{figure}[ht!]
	\centering
		\includegraphics[scale=0.4]{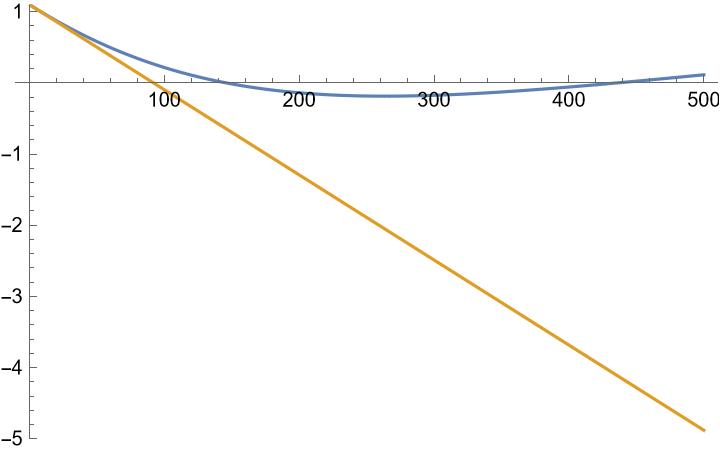} 	
		\caption{\small{Graph of the discriminant $-\Delta_{730119} (r)$ (blue) and $-Z_N (g_{730119} ; r)$ (orange) with $0 \leq r \leq 1$.}}
\label{fig:f8.2}
	\end{figure}
	
	In Fig.~\ref{fig:f8.2}, we observe a stark deviation from Gram's law and its corrected version for \( r=1 \). Unlike prior examples, the discriminant \(\Delta_{730119}(r)\) doesn't maintain positivity over the interval \(0 \leq r \leq 1\). Specifically, it enters a region of negativity. This transition suggests a collision of two consecutive zeros in the vicinity of \(g_n\), as seen in Fig.~\ref{fig:f8.3}:

	\begin{figure}[ht!]
	\centering
		\includegraphics[scale=0.4]{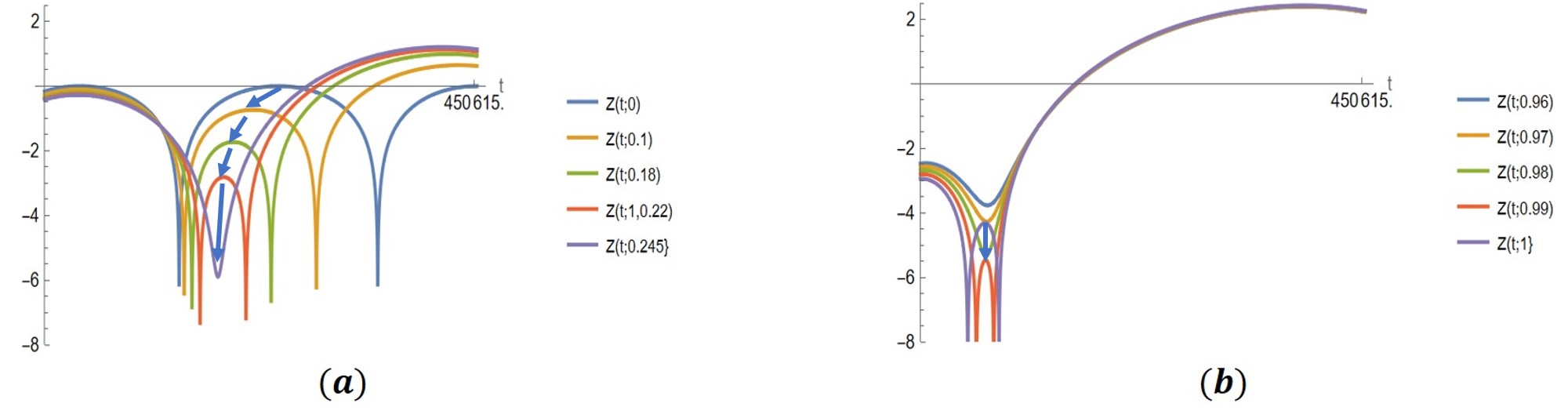} 	
		\caption{\small{$\ln \abs{Z_N(t ; r)}$ for the following values of $r$:  (a) $0$ (blue), $0.1$ (yellow), $0.18$ (green), $0.22$ (red), $0.245$ (purple) and (b) $0.96$ (blue), $0.97$ (yellow), $0.98$ (green), $0.99$ (red), $1$ (purple) for $N=268$ in the range $450613.58 \leq t \leq 450614.8$.}}
\label{fig:f8.3}
	\end{figure} 
	
	The domain of the graph in Fig.~\ref{fig:f8.3} spans from \( t = 450613.58 \) to \( t = 450614.8 \) with \( N = 268 \). The behaviour captured in this figure is crucial for interpreting the discriminant trends observed in Fig.~\ref{fig:f8.2}. Specifically: (1) the collision of zeros, especially visible in (a); (2) the subsequent movement of these zeros, now as complex values, descending along the real axis in areas where the discriminant is negative; and (3) their eventual emergence lower on the real line, illustrated in (b). We will revisit this example when discussing the corrected connecting path in Example~\ref{ex:corrected}.

 \end{ex}

The above example also reveals a peculiar aspect related to Newton's method:
\begin{cor}[Violation of straight-forward Edwards' speculation]
Newton's method, starting from $t^0_n$ for some integer $n$, does not always converge to $t_n$, the $n$-th real zero of $Z(t)$.
\end{cor}

\begin{proof}
 Consider the zeros of the $Z$-function given as:
$$
\begin{array}{ccc} t_{730120}=450613.7144 & ; & t_{730121}=450613.8004. \end{array}
$$ 
If we initiate Newton's method from the point $t^0_{730120}=450613.9648$ we observe that 
	\(
	lim_{k \rightarrow \infty} (t_{730120}^k)=t_{730121} \)
	signifying that the process converges to the adjacent zero, $t_{730121}$, rather than the intended zero, $t_{730120}$. Hence, Newton's method does not always converge to the intended zero.
\end{proof}

This result implies that the straightforward application of Newton's method in identifying the zeros of $Z(t)$ from the zeros of $Z_0(t)$ may not always provide accurate results. It can lead to misidentification of zeros, as seen in the example where the method converged to an adjacent zero instead of the intended zero. Moreover, we see that for $n=730119$ there exists  a region within this interval for which $(-1)^n \Delta_n(r)$ is negative, contrary to previous examples. However, our dynamic version of RH does require us to obtain a one-to-one correspondence between the zeros of $Z_0(t)$ and those of $Z(t)$. In the subsequent section, we will propose a more sensitive approach to ensuring the validity of the corrected Gram's law, based on the A-philosophy assuming the G-B-G repulsion relation.

\section{Correcting the Linear Curve - Overcoming Over-Tasking} \label{s:12}

The corrected Gram's law requires to show that \( (-1)^n \Delta_n(\overline{1}) \) is well-defined and positive. For good Gram points, as well as for the majority of bad Gram points, the linear curve in the \( A \)-parameter space satisfies \( (-1)^n \Delta_n(\overline{r}) > 0 \) for all \( 0 \leq r \leq 1 \), thereby proving suitable for establishing the corrected Gram's law for these points. However, we observed that there exist bad Gram points \( g_n \) where the discriminant \( \Delta_n(\overline{r}) \) for the linear curve may not consistently be non-negative. To affirm the corrected Gram's law for such Gram points, our strategy involves substituting the linear curve with a more sensitive  curve \( \gamma \) in the multi-dimensional \( A \)-space for which \( (-1)^n \Delta_n(r; \gamma) \) is anticipated to be non-negative for all \( 0 \leq r \leq 1 \). In order to describe the construction of $\gamma$ let us note that we observed in Example \ref{ex:g} for $n=730119$ two simultaneous phenomena which occur as $r$ increases:
\begin{enumerate}
\item The point $g_n(r)$ undergoes a sideways shift. This shift is a direct consequence of the viscosity bound, causing the position of $g_n(r)$ to vary laterally.
\item The value of $(-1)^n \Delta_n (r)$ experiences a decrease due to the fact that $g_n$ is characterized as a bad Gram point.
\end{enumerate} 

We observed that the decrease in $(-1)^n \Delta_n(r)$ for the linear curve happens at a faster initial rate than the lateral shift of $g_n(r)$, a condition we refer to as \emph{over-tasking}. This suggests that, in such cases, we should update the connecting curve by arranging for the lateral shift to occur before the decrease towards the axis takes place. Consider Fig. \ref{fig:f7}   

	 \begin{figure}[ht!]
	\centering
		\includegraphics[scale=0.4]{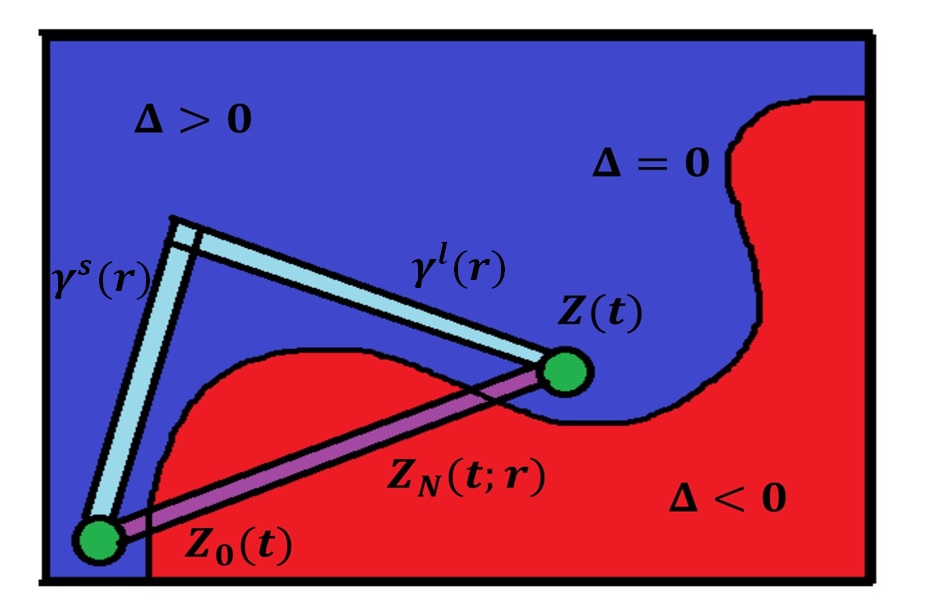} 	
		\caption{\small{The line $Z_N(t;r)$ (purple) in parameter space connecting $Z_0(t)$ to $Z(t)$ in a colliding manner together with the
		corrected broken curve (cyan) connecting the core to $Z(t)$ without passing through the negative region (red).}}
\label{fig:f7}
	\end{figure} 
Figure \ref{fig:f7} suggests a schematic interpretation, in terms of the geometry of parameter space. To put this in more relatable terms, recall the formulation for the linear curve:
\[
Z_N(t;r):= Z_0(t) + r \cdot \sum_{k=1}^{N} A_k(t) \in \mathcal{Z}_N,
\] 
Here, each term $A_k(t)$ is defined as:

\[
A_k(t):= \frac{1}{\sqrt{k+1} } \cos ( \theta (t) - \ln(k+1) t).
\]
In this standard linear curve, all terms $A_k(t)$ see an equal increment as we increase the value of $r$. Once we are willing to take general curves into account, for which each parameter can be adjusted independently, there is a tremendous amount of possibilities arising. Indeed, for the $n$-th discriminant, the dimension of the parameter space is $N(n)$-dimensional which, of course, goes to infinity as $n \rightarrow \infty$. Therefore, how do we determine which curve within this vast parameter space is best suited to adjust the linear curve?

This is the point where our previous local analysis of the discriminant and specifically the repulsion property becomes crucial. Assuming that \( n \in \mathbb{Z} \) is an isolated bad Gram point, according to the repulsion property, we expect that among the indices, there is an enhanced tendency to alter the position of \( g_n ( \overline{a}) \) relative to the \( t \)-axis. This leads us to propose that one can partition the indices ranging from one to \( N(n) \) into two fundamentally distinct classes:

\begin{enumerate}
    \item \emph{Shifting indices:} The index set \( I_{shift} \) which corresponds to terms exerting a significant influence on the lateral shift in \( g_n(r) \).
    \item \emph{Descending indices:} The index set \( I_{descend} \) associated with terms that predominantly contribute to the reduction in \( (-1)^n \Delta_n(r) \).
\end{enumerate}

This allows us to suggest a substantial dimensionality reduction by considering the following $2$-parametric system arising from this split of the indices:

\[ 
\label{eq:2dim}
Z_N(t;r_1,r_2):= Z_0(t) + r_1 \cdot \sum_{k \in I_{shift}} A_k(t)+ r_2 \cdot \sum_{k \in I_{descend}} A_k(t) \in \mathcal{Z}_N,
\] 
We hence get the following simplified version of the corrected Gram's law: 

\begin{conj}[$2$-dim. corrected Gram's law for isolated points] \label{con:2dim} Let $g_n$ be an isolated bad Gram point. Then it is possible to find a curve $\gamma(r) = (r_1(r),r_2(r))$ in the $2$-dimensional parameter space \eqref{eq:2dim} such that $\gamma(0)=(0,0)$, $\gamma(1)=(1,1)$ and \[ \Delta_n( r ; \gamma):=Z_N(g_n (r_1(r), r_2(r)) ;r_1(r),r_2(r) ) >0 \] for all $ 0 \leq r \leq 1$. 
\end{conj}

 In fact, we suggest \( \gamma_{\text{correct}} \subset \mathbb{R}^2 \) should be composed of the following two stages:

\begin{description}
  \item[The shifting stage] This involves increasing the shifting parameters alongside the descending parameters. The goal is to shift \( g_n (\overline{a}) \) while maintaining \( (-1)^n \Delta_n(\overline{a}) \) constant. This stage starts at \( Z_0(t) \) and follows the non-linear level curve of \( (-1)^n \Delta_n (\overline{a}) \) with a value of 1, until an exit point is reached.

  \item[The descending stage] After \( g_n (\overline{a}) \) is approximately in its intended position, all parameters increase linearly towards \( Z_N(t) \). This stage is a linear segment connecting the prior exit point to \( (1,1) \). It's termed the 'descending' stage due to its potential to decrease the value of \( (-1)^n \Delta_n( \overline{a}) \).
\end{description}

For the majority of bad Gram points, the shifting stage is negligible. As a result, the linear and descending curves are approximately identical. Let us thus consider the following example where a substantial shift is required: 

\begin{ex}[Correcting curve for $g_{730119}$] \label{ex:corrected} In order to define the correcting curve for $n=730119$ we should first identify the shifting parameters. As we saw in Fig. \ref{fig:f14} the main surge in this case occurs in the domain where $k$ is between $1$ and $\sqrt{N(n)}=15$.  Table \ref{tab:values} shows the values of $\cos(\ln(k)g_n), \sin(\ln(k)g_n), A_k(g_n)$ and $B_k(g_n)$ within this range: 
\captionsetup{skip=10pt}

\definecolor{LightRed}{rgb}{1,0.8,0.8}
\begin{table}[h!]
    \centering
    \resizebox{\textwidth}{!}{%
        \begin{tabular}{
            |c|
            >{\columncolor{red!20}}c|
            >{\columncolor{red!20}}c|
            c|
            >{\columncolor{red!20}}c|
            c|
            >{\columncolor{red!20}}c|
            c|c|c|c|c|
            >{\columncolor{red!20}}c|
            c|c|c|
        }
            \arrayrulecolor{black}\hline
            \( k \) & 1 & 2 & 3 & 4 & 5 & 6 & 7 & 8 & 9 & 10 & 11 & 12 & 13 & 14 & 15 \\
            \hline
            \(\cos(\ln(k)g_n)\) & -0.14 & 0.25 & 0.96 & -0.53 & 0.99 & -0.20 & 0.41 & 0.88 & 0.77 & -0.99 & 0.03 & 0.21 & 0.94 & 0.95 & -0.85 \\
            \hline
            \(\sin(\ln(k)g_n)\) & 0.99 & 0.97 & 0.28 & 0.85 & -0.11 & 0.98 & -0.91 & -0.48 & 0.64 & -0.11 & -1.0 & 0.98 & 0.33 & 0.30 & -0.53 \\
            \hline
            \( A_k(g_n) \) & -0.099 & 0.14 & 0.48 & -0.24 & 0.41 & -0.074 & 0.14 & 0.29 & 0.24 & -0.30 & 0.0082 & 0.058 & 0.25 & 0.25 & -0.21 \\
            \hline
            \( B_k(g_n) \) & 6.86 & 5.02 & 1.16 & 3.02 & -0.345 & 2.70 & -2.27 & -1.09 & 1.34 & -0.210 & -1.79 & 1.64 & 0.521 & 0.449 & -0.748 \\
            \hline
        \end{tabular}
    }
    \caption{\small{$\cos(\ln(k)g_n), \sin(\ln(k)g_n), A_k(g_n)$ and $B_k(g_n)$ for \( k=1,...,15 \).}}
    \label{tab:values}
\end{table}

By direct computation, the shifting parameters, for which $sin(ln(k) g_n)$ is especially large, are $k=1,2,4,6,12$ and are marked red in the table. consider Fig. \ref{fig:f25}:
 \begin{figure}[ht!]
	\centering
		\includegraphics[scale=0.4]{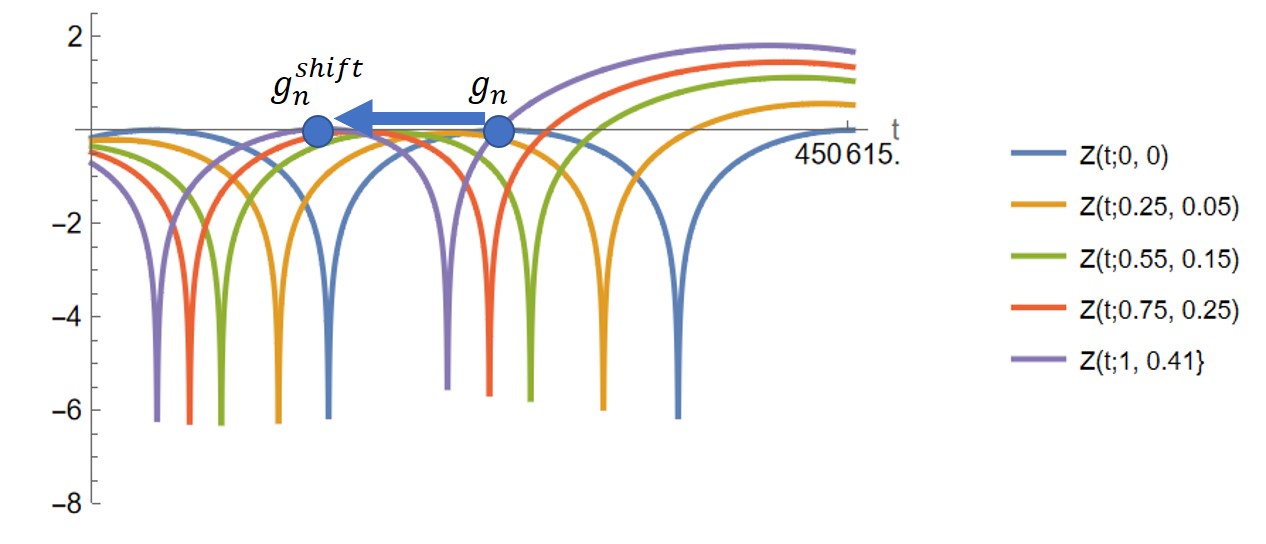} 	
		\caption{\small{Graphs of $\ln \abs{Z_N(t ; r_1 ,r_2)}$ along the following points of the shifting curve:  $(0,0)$ (blue), $(0.25,0.05)$ (yellow), $(0.55,0.15)$ (green), (0.75,0.25) (red), $(1,0.41)$ (purple) for $N=268$ in the range $450613.58 \leq t \leq 450614.8$.}}
\label{fig:f25}
	\end{figure}
	
	 Figure \ref{fig:f25} shows the graphs of $\ln \abs{Z_N(t ; r_1 ,r_2)}$ along the following points $(r_1,r_2)$ of the shifting curve:  $(0,0)$ (blue), $(0.25,0.05)$ (yellow), $(0.55,0.15)$ (green), (0.75,0.25) (red), $(1,0.41)$ (purple) in the range $450613.58 \leq t \leq 450614.8$ and $N=268$. In particular, we see that as $r_1$ transitions from zero (blue) to $1$ (purple) the $n$-th extremal point transitions continuously to the left from $g_n$ to $g_n^{shift}$, while keeping the value of the discriminant fixed. 
	
	The next Fig. \ref{fig:f26} shows the graphs of $\ln \abs{Z_N(t ; r_1 ,r_2)}$ along the following points $(r_1,r_2)$ of the descending curve:  $(1,0.41)$ (blue), $(1,0.6)$ (yellow), $(1,0.8)$ (green), (1,0.95) (red), $(1,1)$ (purple) for $N=268$ in the range $450613.58 \leq t \leq 450614.8$. In this case, as $r_2$ transitions from $0.41$ (blue) to $1$ (purple) the value of the $n$-th extremal point, that is the discriminant, decreases continuously as it transitions from $g^{shift}_n$ to $g_n^{final}$. 
	\begin{figure}[ht!]
	\centering
		\includegraphics[scale=0.4]{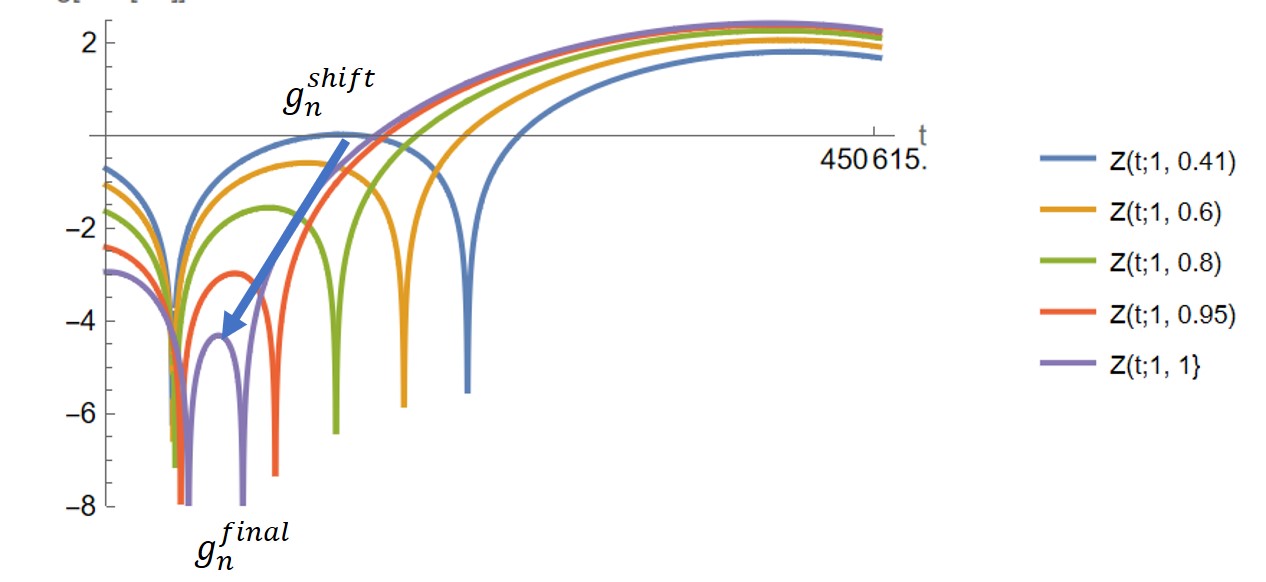} 	
		\caption{\small{graphs of $\ln \abs{Z_N(t ; r_1 ,r_2)}$ along the following points $(r_1,r_2)$ of the descending curve:  $(1,0.41)$ (blue), $(1,0.6)$ (yellow), $(1,0.8)$ (green), (1,0.95) (red), $(1,1)$ (purple) for $N=268$ in the range $450613.58 \leq t \leq 450614.8$.}}
\label{fig:f26}
	\end{figure} 
	
	As we can see, no collision of zero occurs along the corrected connecting curve, which is given by the shifting curve followed by the descending curve. In particular, this avoids the collisions forming for the linear curve, as seen in Example \ref{ex:g}, as required. 
\end{ex}

Let us note the following general remark about the shifting indices: 
\begin{rem} In general, if $k \in I_{shift}$ is a shifting index then an increase in $A_k(t)$ is expected to lead to a shift in $g_n ( \overline{a})$ as well as a local increase in the value of $(-1)^n \Delta_n(r ; \gamma)$, as seen in Table \ref{tab:values} of the above example.  
Indeed, according to Theorem \ref{thm:CM} terms in $k \in I_{shift}$ have high
\[ 
B_k(t):=\frac{1}{\sqrt{k+1} }\ln \left ( \frac{t}{2 \pi (k+1)^2} \right ) \sin ( \theta (t) - \ln(k+1) t).
\]
values. Such terms significantly contribute to $\nabla g_n$, which is also proportional to the Hessian $H_n$. 
Since for these terms $\sin ( \theta (t) - \ln(k+1) t)$ is relatively large, correspondingly $\cos ( \theta (t) - \ln(k+1) t)$ must be  relatively small. As a result, the first-order contribution of the $k$-th term to $\Delta_n(r)$ is minor. 
In particular, this small first-order effect will be negligible with respect to the effect of this $k$-th term on the second-order Hessian. Hence the $k$-th term leads to a shift of $g_n$ as well as a positive inflation in the value of $(-1)^n \Delta_n(r)$.  

\end{rem}

In conclusion, for the case of the Gram point $n=730119$ we saw that:
\begin{enumerate}
 \item The linear curve leads to a collision between the $730119$-th and $730120$-th zeros, occurring when the discriminant $\Delta_n(r)$ vanishes, as seen in Example \ref{ex:g}. 
 \item The corrected connecting curve, composed of its shifting stage and descending stage avoids such collisions, keeping $(-1)^n \Delta_n(r ; \gamma)$ positive for all $0 \leq r \leq 1$, as demonstrated in Example \ref{ex:corrected}.
\end{enumerate} 
 
 Conjecture \ref{con:2dim} proposes that a similar connecting curve exists for any isolated bad Gram point. The key feature is that, by definition, the shifting curve is constructed to leave the discriminant fixed $(-1)^n \Delta_n (r ; \gamma^{shift}) = 1 $, as seen in Fig. \ref{fig:f25}. Hence, it is the role of the decreasing curve to transition the value of $(-1)^n \Delta_n (r ; \gamma^{shift})$ from $1$ to its eventual level for $Z(t)$. This leads us to the following profound conjecture, which we view as the essence of the RH for isolated Gram points: 

\begin{conj}[RH Energy bound for isolated points] \label{con:energy} For any isolated Gram point the following bound holds 
\begin{equation} \label{eq:energy}
(-1)^n \Delta_n (r ; \gamma^{descend}) >0
\end{equation} for all $0 \leq r \leq 1$. 
\end{conj} 

Conjecture \(\ref{con:energy}\) is currently still far out of reach, partially because it intrinsically relies on the G-B-G repulsion as conjectured in Conjecture \(\ref{con:8.1}\). Specifically, since the general separation of the indices into shifting and decreasing is fundamentally based on the infinitesimal G-B-G property, which by itself still requires formal proof as explained in Section \ref{s:10}. We refer to this as an 'energy bound' conjecture to intuitively convey that after the shifting stage has been exhausted, the decreasing curve is hypothesized to lack sufficient 'energy' to cause the discriminant to vanish, or in other words, to lead to a collision of the corresponding consecutive zeros. This is conjectured to be a deeply rooted characteristic of the \(Z\)-functions, playing a pivotal role in governing the corrected Gram's law for isolated bad Gram points. In particular, given the conjectures intimate connection with the fundamental properties of the zeros of the $Z$-function, the formidable challenge of proving it becomes understandable, as its validation would nearly amount to an affirmation of the RH itself. Let us conclude this section with the following three remarks: 

\begin{rem}[The crucial role of the AFE] \label{rem:crucial} 
In Remark \ref{rem:2.1}, we have elaborated on the role of the more robust AFE \eqref{eq:Z-function} which involves a summation of terms up to $[\frac{t}{2}]$, in contrast to the Hardy-Littlewood classical AFE, where the summation involves substantially fewer terms. It is essential to emphasize that the difference between the two equations becomes particularly crucial concerning the energy bound \eqref{eq:energy}. Specifically, there exist Gram points for which the energy bound fails when applied using the Hardy-Littlewood AFE. For this reason, it becomes necessary to define the discriminant as a function of the $A$-parameter space of dimension $[\frac{t}{2}]$, where the higher-order terms seem to act as regulators, preventing the discriminant from vanishing, an aspect that still requires further exploration.
\end{rem}

\begin{rem}[The connecting curve as a non-linear optimization problem]
Conjecture \ref{con:2dim} is concerned with finding a curve connecting $Z_0(t)$ to $Z_N(t)$ in a way that does not pass through the region $\Delta_n( \overline{a})<0$.  A conventional theoretical tool for solving such non-linear optimization issues optimally is the Karush-Kuhn-Tucker theorem. This theorem extends Lagrange multipliers to scenarios involving inequality constraints, as referenced in \cite{Kar,KT}.
If one is willing to ease on optimality and opt for naturality instead, such a curve is, of course, not unique. In particular, our conjectured approach, where a corrected connecting curve comprises shifting and descending stages, while intuitive might not be the exclusive approach. Let us mention that an alternative approach, which also merits investigation, might involve a three-stage construction of the connecting curve along the following lines: 
\begin{enumerate}
    \item Initially, follow the linear curve until the discriminant vanishes. If the discriminant remains non-negative for all \(0 \leq r \leq 1\) no further correction is necessary.
    \item In case of a vanishing discriminant, continue along the discriminant hyper-surface \(\Delta(\overline{a})=0\) until reaching a point of exit, determined by minimization of the Hessian.
    \item Finally, proceed linearly within the \(A\)-parameter space towards \(Z_N(t)\), starting from the point identified in the previous step.
\end{enumerate}
\end{rem}

While our discussion so far has mainly focused on isolated bad Gram points, the RH requires to obtain the corrected Gram's law for all Gram points, including those appearing in Gram blocks of arbitrary length. The motivation for considering the isolated case is that although it can be viewed as the next level of complexity after the trivial case of good Gram points, nevertheless it already reveals the fundamental properties and challenges arising in the general case. In particular, even in the isolated case we are left with the two new fundamental open questions of the repulsion and energy property, elaborated on in this work in Conjecture \ref{con:8.1} and the Conjecture \ref{con:energy}, respectively. In the following last remark we would want to expand on our expectations regarding the extension of our results to the general case: 

\begin{rem}[The corrected Rosser law] Historically, the concept of Gram block's was introduced in a naive attempt to introduce a 'correction of Gram's law', known as Rosser's law, postulating that a Gram block of length $N$ is expected to contain exactly $N$ zeros, see \cite{Ro}. However, as for Gram's law, exceptions to this heuristic have been observed, the first violation occurring around the $13999826$-th Gram point, as shown by Lehman in \cite{Le}. In view of our approach, the reason for this is that Rosser's rule, like Gram's law, is static as it fixes the position of the two good Gram points at its boundaries. In similar lines to the discussion in Section \ref{s:2} we view Rosser's law as a property of the core $Z_0(t)$ rather than that of $Z(t)$. 

 In particular, when taking into the account the dynamic transition from the core $Z_0(t)$ to $Z(t)$ we observe that the block as a whole does not need to be confined to its original position for $Z_0(t)$ and could be shifted altogether. In such a case the corresponding shift in its boundaries needs to be taken into account, as well. The challenge arising for non-isolated Gram blocks is that the individual bad Gram points within the interior of the blocks can actually move in different paces, which adds additional complexity. In particular, in such cases, we postulate that a separate collection of shifting parameters should be allocated for each of the bad Gram points within the block requiring a more sensitive tuning, to avoid collisions. Such Gram blocks will require a dimension reduction to an $N$-dimensional space, generalizing the $2$-dimensional method presented above in the isolated case. Within this space it would be required to identify connecting curves avoiding the union of the discriminant hyper-surfaces of all the points of the block, assuring that no collisions of zeros occur as the block shifts as a whole. Moreover, we suggest that a viscosity bound should be generalized for Gram blocks as a whole. Again, taking into account that fundamental questions still remain open in the isolated case, the full description of the non-isolated case remains out-of-reach at this point. 
\end{rem}

 \section{Summary and Concluding Remarks} \label{s:13}
 The Riemann Hypothesis postulates that all the non-trivial solutions of the equation \( Z(t)=0 \) must be real. In algebraic geometry one has the powerful invariant of the discriminant, which can be seen as a measurement for the realness of zeros of algebraic equations. In this work, we have endeavoured to extend the idea of the discriminant into the transcendental realm of the \( Z(t) \) function. By their very nature, discriminants act as an invariant for a family of functions. Building upon this, we introduced the novel concept of the \( A \)-parametrized space \( \mathcal{Z}_N \) whose elements are given by
\begin{equation}
Z_N(t ; \overline{a}) = \cos(\theta(t)) + \sum_{k=1}^{N} \frac{a_k}{\sqrt{k+1}} \cos \left( \theta (t) - \ln(k+1) t \right),
\end{equation}
where \( \overline{a} = (a_1, \dots, a_N) \in \mathbb{R}^N \) for any \( N \in \mathbb{N} \). In the course of our study, we introduced the local discriminant for a pair of consecutive zeros, \( \Delta_n(\overline{a}) \), defined within the parameter space of dimension \( N(n) := \left [ \frac{\abs{g_n}}{2} \right ] \). This newly defined discriminant has been shown to unveil a wealth of significant new results regarding the zeros of $Z(t)$. To summarize the new results proved in this work:

\begin{enumerate}
    \item We demonstrated that our corrected Gram's law, \( (-1)^n \Delta_n(\overline{1}) > 0 \), is equivalent to the Riemann Hypothesis, as detailed in Theorem \ref{thm:B}.
    
    \item We have shown that the classical Gram's law arises as the first-order approximation of our corrected Gram's law for the linear curve $Z_N(t; r)$ in parameter space, as described in Theorem \ref{thm:first-order}.
    
    \item An examination of the second-order Hessian of our corrected law, in this setting, revealed its relation to shifts of the Gram points along the \( t \)-axis. This connection is elaborated upon in Corollary \ref{Hess-grad}.
    
    \item Based on this discriminant analysis, we identified a previously unobserved numerical repulsion relationship: \( \left| Z'(g_n) \right| > 4 \left| Z(g_n) \right| \) which is observed to hold for isolated bad Gram points. The observed repulsion hints at a natural partitioning of the parameter space into shifting and descending indices, suggesting a dimension reduction for the problem of identifying a corrected connecting curve for isolated Gram points. Consequently, this insight guided us in suggesting an optimization framework aimed at providing a universal validation of the corrected Gram's law, which we have illustrated via various examples.

 \item Our analysis of the Davenport-Heilbronn function highlighted its distinct behaviour when compared with the \( Z \)-function. Notably, we have shown that the Davenport-Heilbronn function does not admit the repulsion property observed for $Z(t)$ in Corollary \ref{cor:DH}. This sheds new light on the elusive question regarding the inherent differences between these two functions.
\end{enumerate}

Our study also unveiled a few fundamental open questions, such as the repulsion relation of Conjecture \ref{con:8.1} and the energy bound of Conjecture \ref{con:energy}. Collectively, the introduction of $\Delta_n(\overline{a})$ together with the conjectures, empirical discoveries and theorems established in this work, contribute to constructing a robust, natural, long sought-after, plausibility argument for the Riemann Hypothesis. Furthermore, they introduce a new dynamical approach for the further study of the zeros of the $Z$-function and related functions.

\bibliographystyle{plain} 
\bibliography{references} 

\end{document}